\documentclass[12pt]{article}
\usepackage{amsfonts, amsmath, amssymb, latexsym, eucal, amscd} 

\usepackage[dvips]{pict2e}

\usepackage{amsthm}
\usepackage{amscd}

\usepackage[dvipdfmx]{graphics}
\usepackage{cite}

\newtheorem{definition}{Definition}[section]
\newtheorem{theorem}[definition]{Theorem}
\newtheorem{lemma}[definition]{Lemma}
\newtheorem{corollary}[definition]{Corollary}
\newtheorem{remark}[definition]{Remark}
\newtheorem{example}[definition]{Example}

\newtheorem{note}[definition]{Note}

\newtheorem{proposition}[definition]{Proposition}

\begin{document} 

\title{\bf
Circular bidiagonal pairs
}
\author{
Paul Terwilliger and Arjana \v{Z}itnik}
\date{}

\maketitle
\begin{abstract} A square matrix is said to be circular bidiagonal whenever (i) each nonzero entry is on the 
diagonal, or the subdiagonal, or in the top-right corner; (ii) each subdiagonal entry is nonzero, and the entry in the top-right corner is nonzero.
Let $\mathbb F$ denote a field, and let $V$ denote a nonzero finite-dimensional vector space over $\mathbb F$.
We consider  an ordered pair of $\mathbb F$-linear maps $A: V \to V$ and $A^*: V \to V$
that satisfy the following two conditions:
\begin{enumerate}
\item[$\bullet$] there exists a basis for $V$ with respect to which the matrix representing $A$ is circular bidiagonal and the matrix representing $A^*$ is diagonal;
\item[$\bullet$] there exists a basis for $V$ with respect to which the matrix representing $A^*$ is circular bidiagonal and the matrix representing $A$ is diagonal.
\end{enumerate}
We call such a pair a circular bidiagonal pair on $V$. 
We classify the  circular bidiagonal pairs up to affine equivalence. There are two infinite families of solutions, which we describe in detail.
\bigskip

\noindent
{\bf Keywords}.  Bidiagonal pair; Hessenberg pair; Leonard pair; tridiagonal pair.
\hfil\break
\noindent {\bf 2020 Mathematics Subject Classification}.
Primary: 17B37;
Secondary: 15A21.
 \end{abstract}


\section{Introduction}
In a celebrated paper \cite{qracah}, Richard Askey and James Wilson introduced the $q$-Racah family of orthogonal polynomials.
In \cite{leonard}, Doug Leonard showed that the $q$-Racah polynomials are the most general orthogonal polynomials that
have orthogonal polynomial duals. In \cite[Theorem~5.1]{bi}, Eiichi Bannai and Tatsuro Ito gave a comprehensive version of Leonard's theorem.
In an effort to clarify and simplify the Leonard theorem, in \cite{2lintrans} the first author introduced the concept of a Leonard pair. Roughly speaking, a Leonard pair consists of two diagonalizable linear maps
on a nonzero finite-dimensional vector space, that each act on an eigenbasis of the other one in an irreducible tridiagonal fashion.
In \cite[Definition~1.4]{2lintrans} there appears an ``oriented''  version of a Leonard pair, called a Leonard system. In \cite[Theorem~1.9]{2lintrans} 
 the Leonard systems are classified up to isomorphism.
The article \cite{notes} contains a modern treatment of this classification, along with a detailed account of the history.
By \cite[Theorem~1.12]{2lintrans} a Leonard pair satisfies two polynomial relations called the tridiagonal relations. 
Some notable papers about Leonard pairs are \cite{ter24, terCanForm, qrac, ter2005, ter2005b, madrid}.
\medskip

\noindent  In \cite{TD00} Tatsuro Ito, Kenichiro Tanabe, and the first author  introduced the concept of a tridiagonal pair as 
a generalization of a Leonard pair.
The concept of a tridiagonal system was also introduced. In \cite[Corollary~18.1]{INT} the tridiagonal systems over an algebraically closed field  are classified up to isomorphism.
In \cite[Section~1.4]{augmented} it is shown how a tridiagonal system induces a tensor product factorization of the underlying vector space.
In \cite{uqsl2hat1, uqsl2hat2} the tridiagonal pairs are related to some finite-dimensional irreducible modules for $U_q(\widehat{\mathfrak{sl}_2})$.
It is shown in \cite[Theorem~10.1]{TD00} that every tridiagonal pair satisfies the tridiagonal relations. 
Some notable papers about tridiagonal pairs are \cite{bockting1, bockting, bocktingQexp, bocktingTer, current,shape, qtet, nom1}.
\medskip

\noindent Over the past 20 years, there appeared in the literature some variations on the Leonard pair and tridiagonal pair concepts. In the next few
paragraphs, we summarize these variations.\medskip

\noindent In \cite{godjali1} Ali Godjali introduced the concept of a Hessenberg pair as a generalization of a tridiagonal pair.
He showed in \cite[Corollary~1.9]{godjali1} that every Hessenberg pair induces a split decomposition of the underlying vector space.
In \cite{godjali2} Godjali considers a special case of Hessenberg pair, called a TH pair. He defines a TH system, and classifies these up to isomorphism \cite[Theorem~6.3]{godjali2}.
In \cite[Section~18]{godjali3} the TH systems are characterized in terms of West/South Vandermonde matrices.
\medskip

\noindent In \cite{funkNeub1} Darren Funk-Neubauer introduced the concept of a bidiagonal pair as a variation on a tridiagonal pair.
 In \cite[Theorem~5.1]{funkNeub1} the bidiagonal pairs are classified up to isomorphism. In 
 \cite[Theorems~5.10, 5.11]{funkNeub1}
  this  classification is interpreted using the equitable presentations of $\mathfrak{sl}_2$ and $U_q(\mathfrak{sl}_2)$.
In \cite{funkNeub2} Funk-Neubauer introduces the concept of a bidiagonal triple. In \cite[Theorem~4.1]{funkNeub2} he shows how every bidiagonal pair extends to a bidiagonal triple. In \cite[Theorem~4.3]{funkNeub2} the bidiagonal triples are classified up to isomorphism.
See \cite{funkNeub3} for related work.
\medskip

\noindent In \cite{vu} Pascal Baseilhac, Azat Gainutdinov, and Thao Vu introduced the concept of a cyclic tridiagonal pair, as a generalization of a tridiagonal pair.
They used cyclic tridiagonal pairs  to study a higher-order generalization of the Onsager algebra.
In \cite[Appendix~A]{vu} some examples of cyclic tridiagonal pairs are given.
It remains an open problem to classify the cyclic tridiagonal pairs up to isomorphism.
\medskip

\noindent In \cite{jhl} Jae-ho Lee introduced the concept of a circular Hessenberg pair. This is a special case of a TH pair, and also a  special case of a cyclic tridiagonal pair.
In \cite[Theorem~5.6]{jhl} the circular Hessenberg pairs are classified  under the assumption that the pair satisfies the tridiagonal relations.
The classification yields four infinite families of solutions \cite[Examples~5.1--5.4]{jhl}.
\medskip

\noindent In the present paper, we introduce the concept of a  circular bidiagonal pair. This is a variation on a bidiagonal pair,
 and a special case of a circular Hessenberg pair. The reason we focus on this special case, is that it affords a classification without
 assuming in advance 
 that the tridiagonal relations are satisfied.
We will display two infinite families of circular bidiagonal pairs. We will introduce the notion of affine equivalence.
Our main result is that every circular bidiagonal pair
is affine equivalent to a member of one of the two families. In the next section, we formally define a circular bidiagonal pair and give
a detailed statement of our results.

\section{Definitions and statement of results}

\noindent In this section, we introduce the concept of a circular bidiagonal pair. To define the concept, we first explain what it means for a square matrix to be circular bidiagonal. The following matrices are
circular bidiagonal:

\begin{align*}
\begin{pmatrix} 3&0&0&1 \\ 1&4&0&0 \\ 0&-1&1&0 \\ 0&0&-1&2 
\end{pmatrix}, \qquad \quad 
\begin{pmatrix} 2&0&0&-1 \\ -1&3&0&0 \\ 0&1&0&0 \\ 0&0&-1&-1 
\end{pmatrix}, \qquad \quad
\begin{pmatrix} 0&0&0&1 \\ 1&0&0&0 \\ 0&1&0&0 \\ 0&0&1&0 
\end{pmatrix}.
\end{align*}

\noindent Circular bidiagonal means (i) each nonzero entry is on the 
diagonal, or the subdiagonal, or in the top-right corner; (ii) each subdiagonal entry is nonzero, and the entry in the top-right corner is nonzero.
\medskip

\noindent  Next, we define a circular bidiagonal pair. For the rest of this paper, $\mathbb F$ denotes a field.

\begin{definition} \label{def:cbp} \rm Let $V$ denote a nonzero vector space over $\mathbb F$ with finite dimension.
By a {\it circular bidiagonal pair} on $V$, we mean an ordered pair of $\mathbb F$-linear maps $A: V \to V$ and $A^*: V \to V$
that satisfy the following two conditions:
\begin{enumerate}
\item[\rm (i)] there exists a basis for $V$ with respect to which the matrix representing $A$ is circular bidiagonal and the matrix representing $A^*$ is diagonal;
\item[\rm (ii)] there exists a basis for $V$ with respect to which the matrix representing $A^*$ is circular bidiagonal and the matrix representing $A$ is diagonal.
\end{enumerate}
\end{definition}
\begin{definition}\rm The circular bidiagonal pair in Definition \ref{def:cbp} is said to be {\it over $\mathbb F$}.
\end{definition}

\begin{definition} \label{def:dual} 
\rm Referring to Definition \ref{def:cbp}, assume that $A, A^*$ is a circular bidiagonal pair on $V$. Then the pair $A^*, A$ is
a circular bidiagonal pair on $V$, called the {\it dual} of $A, A^*$.
\end{definition}

\noindent Next, we give some examples of  circular bidiagonal pairs. Our first example is elementary. Let $V$ denote a vector space over $\mathbb F$ that has dimension one.
Then any ordered pair of $\mathbb F$-linear maps  $A: V \to V$ and $A^*: V \to V$ is a circular bidiagonal pair on $V$. 
\medskip

\noindent
Our next example is more substantial. Consider the vector space $V=\mathbb F^5$ (column vectors). Assume that
$q \in \mathbb F$ is a primitive $5^{\rm th}$ root of unity.
Consider the matrices
\begin{align*}
A= \begin{pmatrix} 0&0&0&0&1 \\ 1&0&0&0&0 \\ 0&1&0&0&0 \\ 0&0&1&0&0 \\ 0&0&0&1&0
\end{pmatrix},
\qquad \qquad A^*= {\rm diag}(1,q,q^2,q^3,q^4). 
\end{align*}
These matrices satisfy 
\begin{align*}
A^5=I, \qquad \qquad (A^*)^5=I, \qquad \qquad A^*A=q A A^*,
\end{align*}
where $I$ denotes the identity matrix.
We claim that the pair $A, A^*$ acts on $V$ as a circular bidiagonal pair. To see this, we check that $A, A^*$
satisfy the conditions in Definition \ref{def:cbp}. The matrix $A$ is circular bidiagonal and the matrix $A^*$ is diagonal. Therefore, condition (i) in Definition \ref{def:cbp} is satisfied by the
basis for $V$ consisting of the columns of $I$. Define a matrix
\begin{align*}
P = \begin{pmatrix} 1&1&1&1&1 \\ 1&q&q^2&q^3&q^4 \\ 1&q^2&q^4&q^6&q^8 \\ 1&q^3&q^6&q^9&q^{12}\\ 1&q^4&q^8&q^{12}&q^{16}
\end{pmatrix}.
\end{align*}
The matrix $P$ is Vandermonde, and hence invertible. One checks that 
$A^* P = P A$. In this equation, take the transpose of each side to obtain $PA^*= A^{-1} P$. Rearranging this equation, we obtain  $A P =  P (A^*)^{-1}$.
These results show that condition (ii) of Definition \ref{def:cbp} is satisfied by the basis for $V$ consisting of the columns of $P$.
We have shown that the pair $A, A^*$ acts on $V$ as a circular bidiagonal pair.
\medskip


\noindent The previous circular bidiagonal pair is a member of an infinite family of circular bidiagonal pairs. Before describing  this family, we bring in some notation.
For the rest of this paper,  every vector space and algebra mentioned is understood to be over $\mathbb F$.
Pick an integer $d\geq 1$. Let ${\rm Mat}_{d+1}(\mathbb F)$ denote the algebra consisting of the $d+1$ by $d+1$ matrices that have all entries in $\mathbb F$.
We index the rows and columns by $0,1,2,\ldots, d$. Let $\mathbb F^{d+1}$ denote the vector space consisting of the column vectors that have $d+1$ coordinates
and all entries in $\mathbb F$. We index the coordinates by $0,1,2,\ldots, d$. Note that ${\rm Mat}_{d+1}(\mathbb F)$ acts on $\mathbb F^{d+1}$ by left multiplication.
Let $I \in {\rm Mat}_{d+1}(\mathbb F)$ denote the identity matrix.


\begin{lemma} \label{lem:qn1} \rm Pick an integer $d\geq 1$, and consider the vector space $V=\mathbb F^{d+1}$. 
 Assume that $q \in \mathbb F$ is a primitive $n^{\rm th}$ root of unity, where $n=d+1$. Define matrices $A, A^* \in {\rm Mat}_{d+1} (\mathbb F)$ as follows.
We have $A_{0,d}=1$, and $A_{i,i-1}=1$ for $1 \leq i \leq d$. All other entries of $A$ are zero.
The matrix $A^*$ is diagonal, with $A^*_{i,i}=q^i$ for $0 \leq i \leq d$. Then the pair $A, A^*$ acts on $V$ as a circular bidiagonal pair. Moreover 
\begin{align}
\label{eq:3r}
A^n =I, \qquad \qquad (A^*)^n=I, \qquad \qquad A^*A=qAA^*.
\end{align}
\end{lemma}
\begin{proof} The relations \eqref{eq:3r} are readily checked.
Define a matrix $P \in {\rm Mat}_{d+1}(\mathbb F)$ that has $(i,j)$-entry $q^{ij}$ for $0 \leq i, j\leq d$. The matrix $P$ is Vandermonde, and hence
invertible. One checks that $A^*P=PA$ and $AP=P (A^*)^{-1}$. Consequently, the pair $A, A^*$ acts on $V$ as a circular bidiagonal pair. 
\end{proof}
\begin{note}\rm The relation on the right in \eqref{eq:3r} is a defining relation for the quantum torus algebra; see for example \cite{gupta}.
\end{note} 
\noindent For the next example, we return to the vector space $V=\mathbb F^5$. Assume that $q \in \mathbb F$ is a primitive $5^{\rm th}$ root of unity.
Pick  $\varepsilon \in \mathbb F$ that is not among $1,q,q^2,q^3,q^4$.
Consider the matrices
\begin{align*}
A= \begin{pmatrix} \varepsilon &0&0&0&1-\varepsilon \\ 1-q^{-1} \varepsilon&q^{-1}\varepsilon&0&0&0 \\ 0&1-q^{-2} \varepsilon& q^{-2}\varepsilon &0&0 \\ 0&0&1-q^{-3} \varepsilon&
q^{-3} \varepsilon&0 \\ 0&0&0&1-q^{-4} \varepsilon& q^{-4}\varepsilon
\end{pmatrix}, \qquad  A^*= {\rm diag}(1,q,q^2, q^3, q^4).
\end{align*}
One checks that 
\begin{align*}
A^5 = I, \qquad \qquad (A^*)^5 = I, \qquad \qquad 
\frac{q A A^* - A^* A}{q-1} = \varepsilon I.
\end{align*}
We will show that the pair $A, A^*$ acts on $V$ as a circular bidiagonal pair.  This is a special case of the following result.

\begin{lemma} \label{ex:ddq} \rm Pick an integer $d\geq 1$, and consider the vector space $V=\mathbb F^{d+1}$. 
 Assume that $q \in \mathbb F$ is a primitive $n^{\rm th}$ root of unity, where $n=d+1$. 
Pick  $\varepsilon \in \mathbb F$ that is not among $1, q,q^2,\ldots, q^d$.
 Define a matrix $A =A(q, \varepsilon)$ in $ {\rm Mat}_{d+1} (\mathbb F)$ as follows.
We have $A_{i,i}=q^{-i}\varepsilon$ for $0 \leq i \leq d$. 
We have $A_{0,d}=1-\varepsilon$, and    $A_{i,i-1}=1-q^{-i}\varepsilon$ for $1 \leq i \leq d$. All other entries of $A$ are zero.
We define a diagonal matrix  $A^*= A^*(q)$ in ${\rm Mat}_{d+1}(\mathbb F)$ with $A^*_{i,i}= q^i$ for $0 \leq i \leq d$.
Then the pair $A, A^*$ acts on $V$ as a circular bidiagonal pair. Moreover 
\begin{align}\label{eq:AAsRel}
A^n = I, \qquad \qquad (A^*)^n = I, \qquad \qquad \frac{qAA^*-A^*A}{q-1} = \varepsilon I.
\end{align}
\end{lemma}
\begin{proof} 
Define a matrix $P=P(q,\varepsilon)$ in ${\rm Mat}_{d+1}(\mathbb F)$ with $(i,j)$-entry
\begin{align} \label{eq:Pmatq}
 P_{i,j} = q^{ij} \frac{(\varepsilon q^{-i}; q)_j}{(\varepsilon q; q)_j} \qquad \qquad (0 \leq i,j\leq d).
\end{align}
The above notation is explained in Section 3.
The following two relations are verified by matrix multiplication:
\begin{align} \label{eq:AP1}
A(q,\varepsilon) P(q, \varepsilon) &= P(q, \varepsilon) A^*(q^{-1}), \\
A^*(q) P(q, \varepsilon) &= P(q, \varepsilon) A(q^{-1}, \varepsilon). \label{eq:AP2}
\end{align}
We claim  that $P(q, \varepsilon)$ is invertible. To prove the claim, we show that
\begin{align}
P(q,\varepsilon) P(q^{-1}, \varepsilon) = \frac{(q;q)_d}{(\varepsilon q;q)_d} I. 
\label{eq:PP}
\end{align}
 Abbreviate $Y=P(q,\varepsilon) P(q^{-1}, \varepsilon)$.
 Observe that \eqref{eq:AP1}, \eqref{eq:AP2}
remain valid if we replace $q$ by $q^{-1}$. By this observation, $Y$ commutes with  $A(q, \varepsilon)$ and $A^*(q)$. 
The matrix $Y$ commutes with $A^*(q)={\rm diag}(1,q,\ldots, q^d)$, so $Y$ is diagonal.
Write $Y= {\rm diag}(y_0, y_1, \ldots, y_d)$. For $1 \leq i \leq d$ we compare the $(i,i-1)$-entry on each side of
 $A(q,\varepsilon) Y = Y A(q, \varepsilon)$; this yields
$ y_{i-1} = y_{i}$. Consequently $y_0=y_1=\cdots = y_d$, so $Y=y_0 I$.
We have
\begin{align}
\label{eq:pp}
P(q,\varepsilon) P(q^{-1}, \varepsilon) = y_0 I.
\end{align}
\noindent For the product on the left in \eqref{eq:pp}, we compute  the $(0,0)$-entry using matrix multiplication, and express the
result in terms of basic hypergeometric series \cite{gr}; this yields
\begin{align*}
y_0 = \sum_{j=0}^d \frac{(\varepsilon;q)_j}{ (\varepsilon q; q)_j}
    =  {}_2 \phi_{1} \biggl( \begin{matrix} q^{-d}, \varepsilon \\ \varepsilon q \end{matrix} \bigg\vert q, 1 \biggr)
    =  \frac{(q;q)_d}{(\varepsilon q;q)_d}.
\end{align*}    
In the above line,  the last equality is the $q$-Vandermonde summation formula \cite[Appendix II]{gr}:
\begin{align*}
 {}_2 \phi_{1} \biggl( \begin{matrix} q^{-d}, b \\ c \end{matrix} \bigg\vert q, \frac{c q^d}{b} \biggr)
    =  \frac{(b^{-1}c;q)_d}{(c;q)_d}
\end{align*}
with $b= \varepsilon$ and $c=\varepsilon q$. We have verified \eqref{eq:PP}, and the claim is proven.
By the claim  and \eqref{eq:AP1}, \eqref{eq:AP2} the pair  $A, A^*$ acts on $V$ as a circular bidiagonal pair. 
Concerning the relations in \eqref{eq:AAsRel}, the last two are verified by  matrix multiplication, and the
first is obtained from the second using \eqref{eq:AP1}.
\end{proof}

\begin{note}\label{note:eps} \rm For the circular bidiagonal pair in Lemma  \ref{ex:ddq},
if we set $\varepsilon=0$ then we get the  circular bidiagonal pair in Lemma \ref{lem:qn1}.
\end{note}

\begin{remark}\rm  Referring to Lemma \ref{ex:ddq}, the number of primitive $n^{\rm th}$ roots of unity depends on $\mathbb F$ and $n$; this number might be zero.
For example, 
if   ${\rm Char}(\mathbb F)$ divides $n$ then $\mathbb F$ does not contain a primitive $n^{\rm th}$ root of unity.
\end{remark}

\begin{definition} \label{def:name1} \rm The circular bidiagonal pair $A, A^*$ in Lemma \ref{ex:ddq} will be called ${\rm CBP}({\rm \mathbb F}; d, q, \varepsilon)$.
\end{definition}



 
\noindent For the next example, we return to the vector space $V=\mathbb F^5$. Assume that ${\rm Char}(\mathbb F)=5$. 
Pick  $\gamma \in \mathbb F$ that is not among $0,1,2,3,4$.
Consider the matrices
\begin{align*}
A= \begin{pmatrix} \gamma &0&0&0&-\gamma \\ -1-\gamma&1+\gamma&0&0&0 \\ 0&-2-\gamma&2+\gamma&0&0 \\ 0&0&-3-\gamma&3+\gamma&0 \\ 0&0&0&-4-\gamma&4+\gamma
\end{pmatrix}, \qquad  A^*= {\rm diag}(0,1,2,3,4).
\end{align*}
One checks that 
\begin{align*}
A^5 = A, \qquad \quad (A^*)^5 = A^*, \qquad \quad
AA^*-A^*A+A-A^*=\gamma I.
\end{align*}
We will show that the pair $A, A^*$ acts on $V$ as a circular bidiagonal pair.  This is a special case of the following result.

\begin{lemma} \label{ex:dd} \rm Pick an integer $d\geq 1$, and consider the vector space $V=\mathbb F^{d+1}$. Assume that $n=d+1$ is prime, and that ${\rm Char}(\mathbb F)=n$. 
Pick  $\gamma \in \mathbb F$ that is not among $0,1, 2,\ldots, d$.
We define a matrix $A=A(\gamma)$ in ${\rm Mat}_{d+1}(\mathbb F)$ as follows. 
We have $A_{i,i}=i+\gamma$ for $0 \leq i \leq d$. 
We have $A_{0,d}=-\gamma$, and $A_{i,i-1}=-i-\gamma $ for $1 \leq i \leq d$. All other entries of $A$ are zero.
We define a diagonal matrix  $A^* \in {\rm Mat}_{d+1}(\mathbb F)$ with $A^*_{i,i}= i$ for $0 \leq i \leq d$.
Then the pair $A, A^*$ acts on $V$ as a circular bidiagonal pair. Moreover 
\begin{align}
\label{eq:gam}
A^n = A, \qquad \quad (A^*)^n = A^*, \qquad \quad AA^*-A^*A+A-A^*=\gamma I.
\end{align}
\end{lemma}
\begin{proof} 
Define a matrix $P=P(\gamma)$ in ${\rm Mat}_{d+1}(\mathbb F)$ with $(i,j)$-entry
\begin{align} \label{eq:Pmat}
 P_{i,j} = \frac{(-i-\gamma)_j}{(1-\gamma)_j}  \qquad \qquad (0 \leq i,j\leq d).
\end{align}
The above notation is explained in Section 3.
The following two relations are verified by matrix multiplication:
\begin{align} \label{eq:APP2}
A(\gamma) P(\gamma ) &= P(\gamma) A^*, \\
A^* P(\gamma) &= P(\gamma) A(-\gamma). \label{eq:AsP2}
\end{align}
\noindent We claim that $P(\gamma)$ is invertible. To prove the claim, we show that
\begin{align} \label{eq:PPi}
P(\gamma ) P(-\gamma) = \frac{d!}{(1-\gamma)_d} I.
\end{align}
Abbreviate $Y=P(\gamma) P(-\gamma)$.
 Observe that \eqref{eq:APP2}, \eqref{eq:AsP2}
remain valid if we replace $\gamma$ by $-\gamma$. By this observation, $Y$ commutes with  $A(\gamma)$ and $A^*$. 
The matrix $Y$ commutes with $A^*={\rm diag}(0,1,2,\ldots, d)$, so $Y$ is diagonal.
Write $Y= {\rm diag}(y_0, y_1, \ldots, y_d)$. For $1 \leq i \leq d$ we compare the $(i,i-1)$-entry on each side of
 $A(\gamma) Y = Y A(\gamma)$; this yields
$ y_{i-1} = y_{i}$. Consequently $y_0=y_1=\cdots = y_d$, so $Y=y_0 I$.
We have
\begin{align}
\label{eq:psp}
P(\gamma) P(-\gamma) = y_0 I.
\end{align}
\noindent For the product on the left in \eqref{eq:psp}, we compute  the $(0,0)$-entry using matrix multiplication, and express the
result in terms of hypergeometric series \cite{sf}; this yields
\begin{align*}
y_0 = \sum_{j=0}^d \frac{(-\gamma)_j}{(1-\gamma)_j }
    =  {}_2 F_{1} \biggl( \begin{matrix} -d, -\gamma \\ 1-\gamma\end{matrix} \bigg\vert  1 \biggr)
    =   \frac{d!}{(1-\gamma)_{d}}.
\end{align*}    
In the above line,  the last equality is the Vandermonde summation formula \cite[Chapter 2]{sf}:
\begin{align*}
 {}_2 F_{1} \biggl( \begin{matrix} -d, b \\ c \end{matrix} \bigg\vert \;1 \biggr)
    =  \frac{(c-b)_d}{(c)_d}
\end{align*}
with $b= -\gamma$ and $c=1-\gamma$. We have verified \eqref{eq:PPi}, and the claim is proven.
By the claim and \eqref{eq:APP2},  \eqref{eq:AsP2} the pair $A, A^*$ acts on $V$ as a circular bidiagonal pair.
Concerning the relations in \eqref{eq:gam}, the last two are verified by matrix multiplication, and the first is obtain from the second using
\eqref{eq:APP2}.
\end{proof}

\begin{definition} \label{def:name2} \rm The circular bidiagonal pair $A, A^*$ in Lemma \ref{ex:dd} will be called ${\rm CBP}(\mathbb F; d, \gamma)$.
\end{definition}

\noindent Next, we define the notion of isomorphism for circular bidiagonal pairs.
\begin{definition}\label{def:cbpIso} \rm Let $A, A^*$ denote a circular bidiagonal pair on a vector space $V$, and let $B, B^*$ denote a circular bidiagonal pair on a vector space $\mathcal V$.
By an {\it isomorphism of circular bidiagonal pairs} from $A, A^*$ to $B, B^*$ we mean an isomorphism of vector spaces $\sigma: V \to {\mathcal V}$
such that $\sigma A = B \sigma $ and $\sigma A^* = B^* \sigma$. We say that the circular bidiagonal pairs $A, A^*$ and $B, B^*$ are {\it isomorphic}
whenever there exists an isomorphism of circular bidiagonal pairs from $A, A^*$ to $B, B^*$.
\end{definition}

\noindent In Section 7, we use the concepts of isomorphism and duality to intrepret the proof of Lemmas \ref{ex:ddq}, \ref{ex:dd}.
\medskip

\noindent Next, we show that the circular bidiagonal pairs in Lemmas \ref{ex:ddq}, \ref{ex:dd} are mutually nonisomorphic.

\begin{lemma} \label{thm:th2} The following {\rm (i), (ii)} hold for $d\geq 1$.
\begin{enumerate}
\item [\rm (i)] 
Assume that ${\rm Char}(\mathbb F) \not=d+1$.
Then circular bidiagonal pairs ${\rm CBP}(\mathbb F; d, q, \varepsilon)$ and ${\rm CBP}(\mathbb F; d, q', \varepsilon')$
are isomorphic if and only if both
\begin{align*}
 q = q', \qquad  \varepsilon = \varepsilon'.
\end{align*}
\item[\rm (ii)] Assume that  ${\rm Char}(\mathbb F) =d+1$.
Then circular bidiagonal pairs ${\rm CBP}(\mathbb F; d, \gamma )$ and ${\rm CBP}(\mathbb F; d, \gamma')$
are isomorphic if and only if $\gamma= \gamma'$.
\end{enumerate}
\end{lemma}
\noindent The proof of Lemma \ref{thm:th2} will be completed in Section 5.
\medskip

\noindent Next, we describe how to adjust a circular bidiagonal pair to obtain another circular bidiagonal pair. 

\begin{lemma} \label{lem:adjust} Let $A, A^*$ denote a circular bidiagonal pair on a vector space $V$. Pick scalars 
$s, s^*, t, t^* $ in $\mathbb F$ with $s, s^*$ nonzero. Then the pair $sA+tI, s^*A^*+t^*I$ is a circular bidigonal pair on $V$.
\end{lemma}
\begin{proof} Routine.
\end{proof}

\begin{definition}\rm Referring to Lemma \ref{lem:adjust}, the pair $sA+tI, s^*A^*+t^*I$ is called an {\it affine transformation} of $A, A^*$.
\end{definition}

\noindent Next, we define the notion of affine equivalence for circular bidiagonal pairs.

\begin{definition} \label{def:aff} \rm Let $A, A^*$ and $B, B^*$ denote circular bidiagonal pairs over $\mathbb F$.
We say that $A, A^*$ and $B, B^*$ are {\it affine equivalent} whenever there exists an affine transformation of $A, A^*$ that is isomorphic to $B, B^*$.
\end{definition}

\noindent Next, we apply the concept of affine equivalence to the circular bidiagonal pairs in Lemmas \ref{ex:ddq}, \ref{ex:dd}.

\begin{lemma} \label{thm:th3} The following {\rm (i), (ii)} hold for $d\geq 1$.
\begin{enumerate}
\item[\rm (i)]
Assume that ${\rm Char}(\mathbb F) \not=d+1$.
Then  circular bidiagonal pairs ${\rm CBP}(\mathbb F; d, q, \varepsilon)$ and ${\rm CBP}(\mathbb F; d, q', \varepsilon')$
are affine equivalent if and only if both
\begin{align*}
 q = q', \qquad \qquad  \varepsilon' \in \lbrace \varepsilon, q\varepsilon, q^2 \varepsilon, \ldots, q^d \varepsilon\rbrace.
\end{align*}
\item[\rm (ii)] Assume that ${\rm Char}(\mathbb F) =d+1$.
Then  circular bidiagonal pairs ${\rm CBP}(\mathbb F; d, \gamma)$ and ${\rm CBP}(\mathbb F; d, \gamma')$
are affine equivalent if and only if
\begin{align*}
\gamma'- \gamma \in \lbrace 0,1,2,\ldots, d\rbrace.
\end{align*}
\end{enumerate}
\end{lemma}
\noindent The proof of Lemma \ref{thm:th3} will be completed in Section 6.
\medskip

\noindent The following is  our main result. 
\begin{theorem} \label{thm:main} Pick an integer $d\geq 1$. Let $A, A^*$ denote a circular bidiagonal pair on a vector space of dimension $d+1$.
First assume that ${\rm Char}(\mathbb F) \not=d+1$. Then 
$A, A^*$ is affine equivalent to ${\rm CBP}(\mathbb F;d,q,\varepsilon)$ for at least one ordered pair $q, \varepsilon$.
Next assume that ${\rm Char}(\mathbb F) =d+1$. Then 
$A, A^*$ is affine equivalent to ${\rm CBP}(\mathbb F; d, \gamma)$ for at least one $\gamma$.
\end{theorem}

\noindent The proof of Theorem \ref{thm:main} will be completed in Section 4.
\medskip

\noindent We have a comment.
\begin{lemma} \label{lem:com} The following {\rm (i), (ii)} hold for $d\geq 1$.
\begin{enumerate}
\item[\rm (i)]  
Assume that ${\rm Char}(\mathbb F) \not=d+1$, and write $A, A^*$ for ${\rm CBP}(\mathbb F; d,q,\varepsilon)$. Then $A, A^*$ is isomorphic to $qA, q^{-1}A^*$.
\item[\rm (ii)]
Assume that ${\rm Char}(\mathbb F) =d+1$, and write $A,A^*$ for ${\rm CBP}(\mathbb F;d,\gamma)$. Then $A, A^*$ is isomorphic to $A-I, A^*-I$.
\end{enumerate}
\end{lemma}

\noindent The proof of Lemma \ref{lem:com} will be completed in Section 8.
\medskip

\noindent
 In Section 9, we discuss how circular bidiagonal pairs are related to the circular Hessenberg pairs
introduced by Jae-ho Lee \cite{jhl}.

\section{Preliminaries}

\noindent
In this section, we review some basic concepts and notation that will be used throughout the paper.
Recall the natural numbers $\mathbb N=\lbrace 0,1,2,\ldots \rbrace$ and
integers $\mathbb Z = \lbrace 0, \pm 1, \pm 2, \ldots \rbrace$.
Recall the field $\mathbb F$ from Section 2. 
  For $a, q  \in \mathbb F$ and $r \in \mathbb N$ define
\begin{align*}
 (a;q)_r = (1-a)(1-aq)(1-aq^2) \cdots (1-aq^{r-1}).
 \end{align*}
 We interpret $(a;q)_0=1$. For $a \in \mathbb F$ and $r \in \mathbb N$ define
\begin{align*}
 (a)_r = a(a+1)(a+2)\cdots (a+r-1).
 \end{align*}
 We interpret $(a)_0=1$.
Let $\lambda$ denote an indeterminate. The algebra $\mathbb F \lbrack \lambda \rbrack$
consists of the polynomials in $\lambda$ that have all coefficients in $\mathbb F$.
Fix an integer $d\geq 1$, and let $V$ denote a vector space with dimension $d+1$.
Let ${\rm End}(V)$ denote the algebra consisting of the $\mathbb F$-linear maps from $V$ to $V$.
Next we recall how each basis $\lbrace v_i \rbrace_{i=0}^d$ of $V$
yields an algebra isomorphism ${\rm End}(V) \to {\rm Mat}_{d+1}(\mathbb F)$.
For $A \in {\rm End}(V)$ and $X \in {\rm Mat}_{d+1}(\mathbb F)$, we say that
{\it $X$ represents $A$ with respect to $\lbrace v_i \rbrace_{i=0}^d$} whenever $Av_j = \sum_{i=0}^d X_{i,j}v_i$ for $0 \leq j \leq d$.
The isomorphism sends $A$ to the unique matrix in ${\rm Mat}_{d+1}(\mathbb F)$ that represents $A$
with respect to $\lbrace v_i \rbrace_{i=0}^d$.
For $A \in {\rm End}(V)$, we say that $A$ is {\it diagonalizable} whenever $V$ is spanned by the eigenspaces of $A$.
We say that $A$ is {\it multiplicity-free} whenever $A$ is diagonalizable, and each eigenspace of $A$ has dimension one.
Assume that $A$ is multiplicity-free, and let $\lbrace V_i \rbrace_{i=0}^d$ denote an ordering of the eigenspaces of $A$.
The sum $V=\sum_{i=0}^d V_i$ is direct.
For $0 \leq i \leq d$ let $\theta_i \in \mathbb F$ denote the eigenvalue of $A$ for $V_i$. By construction, the scalars $\lbrace \theta_i \rbrace_{i=0}^d$
are mutually distinct.
For $0 \leq i \leq d$ define $E_i \in {\rm End}(V)$ such that $(E_i-I)V_i=0$ and $E_iV_j=0$ if $i \not=j$ $(0 \leq j\leq d)$.
Thus $E_i$ is the projection $V \to V_i$.
We call $E_i$ the {\it primitive idempotent} of $A$ associated with $V_i$ (or $\theta_i$). We have
(i) $E_i E_j = \delta_{i,j} E_i$ $(0 \leq i,j\leq d)$;
(ii) $I = \sum_{i=0}^d E_i$;
(iii) $V_i = E_iV$ $(0 \leq i \leq d)$;
(iv) ${\rm tr}(E_i) = 1$ $(0 \leq i \leq d)$;
(v) $A=\sum_{i=0}^d \theta_i E_i$;
(vi) $AE_i = \theta_i E_i = E_i A$ $(0 \leq i \leq d)$.
Moreover
\begin{align}        \label{eq:defEi}
  E_i=\prod_{\stackrel{0 \leq j \leq d}{j \neq i}}
          \frac{A-\theta_jI}{\theta_i-\theta_j} \qquad \qquad (0 \leq i \leq d).
\end{align}
Let $M$ denote the subalgebra of ${\rm End}(V)$ generated by $A$. The vector space $M$ has a basis $\lbrace A^i\rbrace_{i=0}^d$, and also $0 = \prod_{i=0}^d (A-\theta_i I)$.
Moreover, the elements
 $\lbrace E_i \rbrace_{i=0}^d$ form a basis for the vector space $M$.
 Pick scalars $s,t \in \mathbb F$ with $s\not=0$. The map $sA+tI$ is multiplicity-free, with eigenvalues $\lbrace s \theta_i + t\rbrace_{i=0}^d$.
 For $0 \leq i \leq d$, the map $E_i$ is the primitive idempotent of $sA+tI$ associated with  $s\theta_i + t$.
 Abbreviate $n=d+1$. A scalar $q \in \mathbb F$ is called a {\it primitive $n^{\rm th}$ root of unity} whenever $q^n=1$ and $q^i \not=1$
 for $1 \leq i \leq d$. If $q \in \mathbb F$ is a primitive $n^{\rm th}$ root of unity, then in the algebra ${\mathbb F}\lbrack \lambda \rbrack$,
 \begin{align*}
 \lambda^n  - 1 = (\lambda-1)(\lambda-q) \cdots (\lambda- q^d).
 \end{align*}
 If ${\rm Char}({\mathbb F})= n$, then  in the algebra ${\mathbb F}\lbrack \lambda \rbrack$,
 \begin{align*}
 \lambda^n  - \lambda = \lambda(\lambda-1)(\lambda-2) \cdots (\lambda- d).
 \end{align*}
 This fact is a version of Fermat's little theorem \cite[Theorem 1.50]{rotman}.
  For $i,j \in \mathbb Z$ we say that $i\equiv j$ (mod $n$) whenever $n$ divides $i-j$.



\section{The proof of Theorem \ref{thm:main}}

\noindent In this section, our goal  is to prove Theorem \ref{thm:main}.
Throughout this section, we fix an integer $d\geq 1$, a vector space $V$ with dimension $d+1$, and a circular bidiagonal pair 
$A, A^*$  on $V$. 
\medskip

\noindent The following result is a special case of \cite[Lemma~2.1]{godjali2}; we will give a short proof for the sake of completeness.
\begin{lemma} \label{lem:mf} {\rm (See \cite[Lemma~2.1]{godjali2}.)}  Each of $A, A^*$ is multiplicity-free.
\end{lemma}
\begin{proof} We first consider $A$. The map $A$ is diagonalizable by Definition \ref{def:cbp}(ii);  we show that each eigenspace of $A$
has dimension one. To do this, it suffices to show that $A$ has $d+1$ eigenspaces. Let $\lbrace v_i \rbrace_{i=0}^d$
denote a basis for $V$ that satisfies Definition \ref{def:cbp}(i). Let the matrix $B \in {\rm Mat}_{d+1}(\mathbb F)$  represent $A$
with respect to $\lbrace v_i \rbrace_{i=0}^d$. By construction, $B$ is circular bidiagonal. In particular, for $B$ each entry on the
subdiagonal is nonzero and each entry below the subdiagonal is zero. For $0 \leq r \leq d$ we examine the entries of $B^r$.
For $0 \leq i,j\leq d$ the $(i,j)$-entry of $B^r$ is nonzero if $i-j=r$, and zero if $i-j>r$. Therefore, the matrices $\lbrace B^r \rbrace_{r=0}^d$
are linearly independent. By this and linear algebra, the maps $\lbrace A^r \rbrace_{r=0}^d$ are linearly independent. Consequently,
the minimal polynomial of $A$ has degree $d+1$. This minimal polynomial has no repeated roots, since
 $A$ is diagonalizable. Therefore, $A$ has $d+1$ distinct eigenvalues and hence $d+1$ eigenspaces. We have shown that $A$ is multiplicity-free. One similarly shows that
 $A^*$ is multiplicity-free.
\end{proof}

\begin{definition} \label{def:MMs} \rm Let $M$ (resp. $M^*$) denote the subalgebra of ${\rm End}(V)$ generated by $A$ (resp. $A^*$).
\end{definition}
\noindent Note that $\lbrace A^i \rbrace_{i=0}^d$
is a basis for $M$, and  $\lbrace (A^*)^i \rbrace_{i=0}^d$
is a basis for $M^*$. 
\medskip

\begin{definition} \label{def:standard} \rm Let $\lbrace E_i \rbrace_{i=0}^d$ 
(resp. $\lbrace E^*_i \rbrace_{i=0}^d$)
denote an ordering of the primitive idempotents of $A$ (resp. $A^*$). For $0 \leq i \leq d$
let $0 \not=v_i \in E_iV$ and $0 \not=v^*_i \in E^*_iV$. Note that $\lbrace v_i\rbrace_{i=0}^d$ (resp. $\lbrace v^*_i\rbrace_{i=0}^d$) is a basis for $V$. The ordering $\lbrace E_i \rbrace_{i=0}^d$  
(resp. $\lbrace E^*_i \rbrace_{i=0}^d$)
is
called {\it standard} whenever the basis $\lbrace v_i \rbrace_{i=0}^d$ (resp. $\lbrace v^*_i \rbrace_{i=0}^d$) satisfies Definition \ref{def:cbp}(ii) (resp. Definition \ref{def:cbp}(i)).
\end{definition}

\noindent Next we explain how the  standard orderings in Definition \ref{def:standard} are not unique.
In this explanation, we discuss the primitive idempotents of $A$; a similar discussion applies to the primitive idempotents of $A^*$.

\begin{definition}\label{def:to} Let $E$ and $F$ denote  primitive idempotents of $A$.  Let us write $E\to F$ whenever there exists
$\alpha \in \mathbb F$ such that $(A^*-\alpha I) EV=FV$.
\end{definition}

\begin{lemma} \label{lem:suc} For every primitive idempotent $E$ of $A$, there exists a unique primitive idempotent $F$ of $A$ such that
 $E\to F$. Moreover $E \not=F$.
\end{lemma}
\begin{proof} Since $A^*$ acts on the eigenspaces of $A$ in a  circular bidiagonal fashion.
\end{proof}

\begin{lemma} \label{lem:cycle} Let $\lbrace E_i \rbrace_{i=0}^d$ denote an ordering of the primitive idempotents of $A$. This ordering is standard if and only if
$E_0 \to E_1 \to \cdots \to E_d \to E_0$.
\end{lemma}
\begin{proof} By Definitions \ref{def:cbp} and \ref{def:standard}.
\end{proof}

\begin{lemma} There are exactly $d+1$ standard orderings of the primitive idempotents of $A$. 
\end{lemma}
\begin{proof}  By Lemmas \ref{lem:suc} and \ref{lem:cycle}, for every primitive idempotent $E$ of $A$, there exists a unique standard ordering $\lbrace E_i \rbrace_{i=0}^d$
of the primitive idempotents of $A$ such that $E=E_0$. The result follows.
\end{proof}


\begin{definition} \label{def:eigval} \rm For the rest of this section, we fix a standard ordering $\lbrace E_i \rbrace_{i=0}^d$ of the primitive
idempotents of $A$, and a standard ordering $\lbrace E^*_i\rbrace_{i=0}^d$ of the primitive idempotents of $A^*$.
 For $0 \leq i \leq d$ let $\theta_i$ (resp. $\theta^*_i$) denote the eigenvalue of $A$ (resp. $A^*$) for $E_i$ (resp. $E^*_i$).
\end{definition}
\noindent Note that $\lbrace E_i \rbrace_{i=0}^d$
is a basis for $M$, and  $\lbrace E^*_i \rbrace_{i=0}^d$
is a basis for $M^*$. 
\medskip

\noindent We have a comment about the subscript $i$  in $E_i$, $E^*_i$, $\theta_i$, $\theta^*_i$.
Due to the circular nature of a circular bidiagonal pair, our calculations involving these subscripts will be carried out modulo $n$, where $n=d+1$. The details are explained in the following definition.
\begin{definition} \rm
For $X \in \lbrace E, E^*, \theta, \theta^*\rbrace$ and $ i \in \mathbb Z$, we define $X_i=X_r$ where $0 \leq r \leq d$ and $i \equiv r$ (mod $n$).
\end{definition}

\begin{definition} \label{def:ai} \rm For $0 \leq i \leq d$ define
\begin{align}
\label{eq:ai}
 a_i = {\rm tr}(AE^*_i), \qquad \qquad a^*_i = {\rm tr}(A^*E_i).
 \end{align}
\end{definition}

\begin{lemma} \label{lem:trace} The following $\rm (i), (ii)$ hold for $0 \leq i \leq d$:
\begin{enumerate}
\item[\rm (i)]
${\rm tr} (A E^*_i) = {\rm tr}(E^*_i A E^*_i) = {\rm tr}(E^*_i A)$;
\item[\rm (ii)]
${\rm tr} (A^* E_i) = {\rm tr}(E_i A^* E_i) = {\rm tr}(E_i A^*)$.
\end{enumerate}
\end{lemma}
\begin{proof} (i) By linear algebra, ${\rm tr}(XY)= {\rm tr}(YX)$ for all $X, Y \in {\rm End}(V)$. The result follows from this and $(E^*_i)^2 = E^*_i$.
\\
\noindent (ii) Similar to the proof of (i).
\end{proof}

\begin{lemma} \label{lem:EAEa} For $0 \leq i \leq d$ we have
\begin{align*}
E^*_i A E^*_i = a_i E^*_i, \qquad \qquad  E_i A^* E_i = a^*_i E_i.
\end{align*}
\end{lemma}
\begin{proof} We verify the equation on the left. Abbreviate $\mathcal A = {\rm End}(V)$. The primitive idempotent $E^*_i$ has rank one, so 
$E^*_i$ is a basis for $E^*_i {\mathcal A} E^*_i$. Therefore, there exists $\alpha_i \in \mathbb F$ such that $E^*_i A E^*_i = \alpha_i E^*_i$.
In this equation, take the trace of each side and use \eqref{eq:ai} along with Lemma \ref{lem:trace}(i) and  ${\rm tr}(E^*_i)=1$ to obtain $a_i = \alpha_i$. We have
verified the equation on the left. The equation on the right is similarly verified.
\end{proof}

\begin{lemma} \label{lem:aiCirc} For $0 \leq i \leq d$ we have
\begin{align*}
(A- a_i I) E^*_iV = E^*_{i+1}V, \qquad \qquad (A^*- a^*_i I) E_iV= E_{i+1}V.
\end{align*}
\end{lemma}
\begin{proof} We verifiy the equation on the left.  By Definition \ref{def:to} and Lemma \ref{lem:cycle}, there exists $\alpha_i \in \mathbb F$
such that  $(A-\alpha_i I ) E^*_iV=E^*_{i+1}V$. In this equation,  apply $E^*_i$ to each side and evaluate the result using
Lemma \ref{lem:EAEa};  this yields
\begin{align*}
0 = E^*_i (A-\alpha_i I) E^*_iV = (a_i - \alpha_i) E^*_iV.
\end{align*}
Of course $E^*_iV \not=0$, so $\alpha_i = a_i$. We have verified the equation on the left. The equation on the right is similarly verified.
\end{proof}

\begin{lemma} \label{lem:EAE} The following $\rm (i), (ii)$ hold for $0 \leq i,j\leq d$.
\begin{enumerate}
\item[\rm (i)]  ${\displaystyle{
E^*_iAE^*_j = \begin{cases}
 0, & {\mbox{\rm if $i-j \not\in \lbrace 0,1\rbrace$ (mod $n$)}}; \\
\not=0, &{\mbox{\rm if $i-j \equiv 1$ (mod $n$)}}.
\end{cases}
}}$
\item[\rm (ii)] 
 ${\displaystyle{
E_iA^*E_j = \begin{cases}
 0, & {\mbox{\rm if $i-j \not\in \lbrace 0,1\rbrace$ (mod $n$)}}; \\
\not=0, &{\mbox{\rm if $i-j \equiv 1$ (mod $n$)}}.
\end{cases}
}}$
\end{enumerate}
\end{lemma}
\begin{proof} By Lemma \ref{lem:aiCirc}.
\end{proof}

\noindent The  following generalization of Lemma \ref{lem:EAE} will be useful.

\begin{lemma} \label{lem:higher} The following $\rm (i), (ii)$ hold for $0 \leq i,j, r\leq d$.
\begin{enumerate}
\item[\rm (i)]  ${\displaystyle{
E^*_iA^rE^*_j = \begin{cases}
 0, & {\mbox{\rm if $i-j \not\in \lbrace 0,1,\ldots, r\rbrace$ (mod $n$)}}; \\
\not=0, &{\mbox{\rm if $i-j \equiv r$ (mod $n$)}}.
\end{cases}
}}$
\item[\rm (ii)] 
 ${\displaystyle{
E_i(A^*)^rE_j = \begin{cases}
 0, & {\mbox{\rm if $i-j \not\in \lbrace 0,1, \ldots, r\rbrace$ (mod $n$)}}; \\
\not=0, &{\mbox{\rm if $i-j \equiv r$ (mod $n$)}}.
\end{cases}
}}$
\end{enumerate}
\end{lemma} 
\begin{proof} This is a routine consequence of Lemma \ref{lem:EAE}.
\end{proof}

\begin{lemma} \label{lem:prodnz} The following holds for $0 \leq i,j\leq d$:
\begin{enumerate}
\item[\rm (i)] $E_i E^*_j \not=0$;
\item[\rm (ii)] $E^*_i E_j \not=0$.
\end{enumerate}
\end{lemma}
\begin{proof} (i)
Using     \eqref{eq:defEi} and Lemma \ref{lem:higher}(i),
\begin{align*}
E^*_{j+d} E_i E^*_j &= E^*_{j+d} \Biggl( \prod_{\stackrel{0 \leq \ell \leq d}{\ell \neq i}}
          \frac{A-\theta_\ell I}{\theta_i-\theta_\ell} \Biggr)  E^*_j  
      = E^*_{j+d} A^d E^*_j \prod_{\stackrel{0 \leq \ell \leq d}{\ell \neq i}}
          \frac{1}{\theta_i-\theta_\ell}
      \not=0.
\end{align*}
\noindent Therefore $E_i E^*_j\not=0$.
\\
\noindent (ii) Similar to the proof of (i).
\end{proof}

\begin{lemma}\label{lem:basisEE} In each of {\rm (i)--(iv)} below, we give a basis for the vector space ${\rm End}(V)$:
\begin{enumerate}
\item[\rm (i)] $E_i E^*_j$ $(0 \leq i,j\leq d)$;
\item[\rm (ii)] $A^i (A^*)^j $ $(0 \leq i,j\leq d)$;
\item[\rm (iii)] $E^*_i E_j $ $(0 \leq i,j\leq d)$;
\item[\rm (iv)] $(A^*)^i A^j$ $(0 \leq i,j\leq d)$.
\end{enumerate}
\end{lemma}
\begin{proof} (i) The dimension of ${\rm End}(V)$ is $(d+1)^2$, and this is the number of vectors listed. Therefore, it suffices to show that the
listed vectors are linearly independent. Assume that
\begin{align*}
0 = \sum_{i=0}^d \sum_{j=0}^d \alpha_{i,j} E_i E^*_j \qquad \qquad (\alpha_{i,j} \in \mathbb F).
\end{align*}
\noindent We show that $\alpha_{r,s} =0$ for $0 \leq r,s\leq d$. Let $r,s$ be given. We have
\begin{align*}
0 = E_r \Biggl( \sum_{i=0}^d \sum_{j=0}^d \alpha_{i,j} E_i E^*_j \Biggr) E^*_s  = \alpha_{r,s} E_r E^*_s.
\end{align*}
We have $E_r E^*_s \not=0$ by Lemma \ref{lem:prodnz}(i), so $\alpha_{r,s}=0$.\\
\noindent (ii)  By (i) and the notes below Definitions \ref{def:MMs}, \ref{def:eigval}.
\\
\noindent (iii), (iv)  Similar to the proof of (i), (ii).
\end{proof}

\noindent The next three lemmas contain results  about $A$ and $\lbrace E^*_i \rbrace_{i=0}^d$; similar results hold for $A^*$ and $\lbrace E_i\rbrace_{i=0}^d$.

\begin{lemma} \label{lem:BBasis} Let $\theta$ denote an eigenvalue of $A$, and let $0 \not=\xi\in V$ denote a corresponding eigenvector.
Then the following {\rm (i)--(iii)} hold:
\begin{enumerate}
\item[\rm (i)] the vector $E^*_i \xi $ is a basis for $E^*_iV$ $(0 \leq i \leq d)$;
\item[\rm (ii)] the vectors $\lbrace E^*_i \xi \rbrace_{i=0}^d$ form a basis for $V$;
\item[\rm (iii)] the basis $\lbrace E^*_i \xi \rbrace_{i=0}^d$ satisfies Definition \ref{def:cbp}(i).
\end{enumerate}
\end{lemma}
\begin{proof} (i) The dimension of $E^*_iV$ is one, so it suffices to show that $E^*_i \xi \not=0$.
There exists an integer $j$ $(0 \leq j \leq d)$ such that $\theta = \theta_j$. The subspace $E_jV$ has dimension one and contains $\xi$, so $\xi$ spans $E_jV$. Therefore, $E^*_i \xi $ spans $E^*_i E_jV$.
We have $E^*_i E_j \not=0$ by Lemma \ref{lem:prodnz}(ii), so $E^*_i E_j V \not=0$. By these comments $E^*_i \xi \not=0$. \\
\noindent (ii) By (i) and since the sum $V= \sum_{i=0}^d E^*_iV$ is direct. \\
\noindent (iii) By Definition  \ref{def:eigval}.
\end{proof}

\begin{lemma} \label{lem:BBasis2} We refer to the basis  $\lbrace E^*_i \xi \rbrace_{i=0}^d$ in Lemma \ref{lem:BBasis}.
Let $B$ (resp. $B^*$)
denote the matrix in ${\rm Mat}_{d+1}(\mathbb F)$ that represents $A$ (resp. $A^*$) with respect to $\lbrace E^*_i \xi \rbrace_{i=0}^d$.
Then the following {\rm (i)--(iv)} hold:
\begin{enumerate}
\item[\rm (i)] $B$ is circular bidiagonal with constant row sum $\theta$;
\item[\rm (ii)]  $B_{i,i}=a_i$ for $0 \leq i \leq d$;
\item[\rm (iii)] $B_{0,d}=\theta-a_0$,
and $B_{i,i-1}=\theta-a_i$ for $1 \leq i \leq d$;
\item[\rm (iv)] $B^*$ is diagonal, with $B^*_{i,i}=\theta^*_i$ for $0 \leq i \leq d$.
\end{enumerate}
\end{lemma}
\begin{proof} (i) The matrix $B$ is circular bidiagonal by Definition \ref{def:cbp}(i) and Lemma \ref{lem:BBasis}(iii). Define a vector
${\bf 1} \in \mathbb F^{d+1}$ that has all entries 1. We have $B {\bf 1} = \theta {\bf 1}$,  because $\xi = \sum_{i=0}^d E^*_i \xi$ is an eigenvector for $A$ with eigenvalue $\theta$.
By $B {\bf 1} = \theta {\bf 1}$, the matrix $B$ has constant row sum $\theta$.
 \\
\noindent (ii) By Lemma \ref{lem:EAEa}. \\
\noindent (iii) By (i), (ii) above. \\
\noindent (iv) The matrix $B^*$ is diagonal by Definition \ref{def:cbp}(i) and Lemma \ref{lem:BBasis}(iii).
 By Definition \ref{def:eigval} we obtain $B^*_{i,i}=\theta^*_i$ for $0 \leq i \leq d$.
\end{proof}

\begin{lemma} \label{lem:match} Let $\theta$ denote an eigenvalue of $A$. Then $\theta \not= a_i$ for $0 \leq i \leq d$.
\end{lemma}
\begin{proof} Let $0 \not=\xi \in V$ denote an eigenvector for $A$ with eigenvalue $\theta$, and let $B \in {\rm Mat}_{d+1}(\mathbb F)$
represent $A$ with respect to the basis $\lbrace E^*_i \xi \rbrace_{i=0}^d$. The matrix $B$ is circular bidiagonal by
 Lemma \ref{lem:BBasis2}(i). Therefore, 
$B_{0,d} \not=0$ and $B_{i,i-1} \not=0$ for $1 \leq i \leq d$. The result follows in view of Lemma \ref{lem:BBasis2}(iii).
\end{proof}

\noindent Our next general goal is to obtain a relation involving $A$ and $A^*$.
\medskip

\noindent The next two lemmas contain results about $A$ and $\lbrace E^*_i \rbrace_{i=0}^d$; similar results hold for $A^*$ and $\lbrace E_i \rbrace_{i=0}^d$.
\begin{lemma} \label{lem:AE} The following $\rm (i), (ii) $ hold for $0 \leq i \leq d$:
\begin{enumerate}
\item[\rm (i)]  $E^*_i A = E^*_i A E^*_i + E^*_i A E^*_{i-1}$;
\item[\rm (ii)] $A E^*_i  = E^*_i A E^*_i + E^*_{i+1} A E^*_i$.
\end{enumerate}
\end{lemma}
\begin{proof} (i) Using Lemma \ref{lem:EAE}(i) we have
\begin{align*}
E^*_i A = E^*_i A I = \sum_{j=0}^d E^*_i A E^*_j = E^*_i A E^*_i + E^*_i A E^*_{i-1}.
\end{align*}
 (ii) Using Lemma \ref{lem:EAE}(i) we have
\begin{align*}
A E^*_i  = I A E^*_i  = \sum_{j=0}^d E^*_jA E^*_i = E^*_i A E^*_i + E^*_{i+1} A E^*_i.
\end{align*}
\end{proof}

\begin{lemma} \label{lem:AEandEA} For $0 \leq i \leq d$,
\begin{align*}
A E^*_i - a_i E^*_i = E^*_{i+1} A E^*_i = E^*_{i+1} A - a_{i+1} E^*_{i+1}.
\end{align*}
\end{lemma}
\begin{proof} The equation on the left follows from Lemma \ref{lem:EAEa} and Lemma \ref{lem:AE}(ii). The equation on the
right follows from  Lemma \ref{lem:EAEa} and Lemma \ref{lem:AE}(i).
\end{proof} 

\noindent We bring in some notation. Define
\begin{align*}
M^* A M^* = {\rm Span} \lbrace X A Y \vert X, Y \in M^* \rbrace.
\end{align*}

\begin{proposition} \label{prop:MAM}  The following {\rm (i), (ii)} hold.
\begin{enumerate}
\item[\rm (i)] $M^* A M^* + M^* = A M^* + M^*$;
\item[\rm (ii)]  $M^* A M^* + M^* = M^* A + M^*$.
\end{enumerate}
\end{proposition}
\begin{proof} (i) To obtain the inclusion $\subseteq $, we use Lemmas  \ref{lem:EAEa}, \ref{lem:EAE}(i), \ref{lem:AEandEA} to obtain
\begin{align*}
M^* A M^* &= {\rm Span} \lbrace  E^*_i A E^*_j \vert 0 \leq i,j\leq d\rbrace \\
        &= {\rm Span} \lbrace  E^*_i A E^*_j \vert 0 \leq i,j\leq d, \; i-j \in \lbrace 0,1\rbrace \; \mbox{\rm (mod $n$)}\rbrace \\
& = {\rm Span} \lbrace E^*_i A E^*_i \vert 0 \leq i \leq d \rbrace+ {\rm Span} \lbrace E^*_{i+1} A E^*_i \vert 0 \leq i \leq d\rbrace \\
& \subseteq A M^* + M^*.
\end{align*}
\noindent The inclusion $\supseteq$ holds since $I \in M^*$.
\\
\noindent (ii) Similar to the proof of (i).
\end{proof} 

\begin{proposition} \label{prop:AAs}
There exists a unique sequence $q, \alpha, \beta, \gamma $ of scalars in $\mathbb F$ such that
\begin{align}
q A A^* - A^*A+ \alpha A -\beta A^* = \gamma I.
\label{lem:five}
\end{align}
\end{proposition}
\begin{proof} First, we show that the sequence $q, \alpha, \beta, \gamma$ exists.
Using Proposition \ref{prop:MAM}, we obtain
\begin{align*}
A^* A \in  M^*A \subseteq M^*A+M^*=  AM^*  + M^*.
\end{align*}
Therefore, there exist $X, Y \in M^*$ such that $A^* A =A X  + Y$. The elements $\lbrace (A^*)^r \rbrace_{r=0}^d$ form a basis for $M^*$. Write
\begin{align*}
X = \sum_{r=0}^d \alpha_r (A^*)^r, \qquad \qquad Y = \sum_{r=0}^d \beta_r (A^*)^r, \qquad \qquad \alpha_r, \beta_r \in \mathbb F.
\end{align*}
We claim that $\alpha_r = 0 $ and $\beta_r=0$ for $2 \leq r \leq d$.  To prove the claim, we assume that it is false, and get a contradiction.
There exists an integer $r$ $(2 \leq r \leq d)$ such that $\alpha_r \not=0$ or $\beta_r \not=0$.
Define
\begin{align*}
m = {\rm max} \lbrace r \vert 2 \leq r \leq d, \; \alpha_r\not=0 \;{\rm or} \;\beta_r \not=0\rbrace.
\end{align*}
\noindent By construction $2 \leq m \leq d$. Also, $\alpha_r=0$ and $\beta_r=0$ for $m+1 \leq r \leq d$. Therefore
\begin{align*}
X = \sum_{r=0}^m \alpha_r (A^*)^r, \qquad \qquad Y = \sum_{r=0}^m \beta_r (A^*)^r.
\end{align*}
\noindent Let $0 \leq i \leq d$. 
By Lemma \ref{lem:higher}(ii) and the construction, we have
\begin{align*}
E_{m+i} X E_i = \alpha_m E_{m+i} (A^*)^m E_i, \qquad \qquad
E_{m+i} Y E_i = \beta_m E_{m+i} (A^*)^m E_i.
\end{align*}
\noindent Also, by Lemma \ref{lem:EAE}(ii) and $m \geq 2$, we have $E_{m+i} A^* E_i=0$. 
\noindent  We may now argue
\begin{align*}
0 &=  E_{m+i }  A^* E_i \theta_i\\
&=  E_{m+i} A^*A E_i  \\
&= E_{m+i} (AX+Y) E_i \\
&= \theta_{m+i} E_{m+i}  X E_i + E_{m+i} Y E_i \\
&= (\theta_{m+i} \alpha_m + \beta_m) E_{m+i} (A^*)^m E_i.
 \end{align*}
We have $E_{m+i} (A^*)^m E_i \not=0$ by Lemma \ref{lem:higher}(ii). By these comments
$0 = \theta_{m+i} \alpha_m+\beta_m$ for $0 \leq i \leq d$.  In particular, 
\begin{align*}
0 = \theta_0 \alpha_m+\beta_m, \qquad \qquad 0 = \theta_1 \alpha_m+\beta_m.
\end{align*}
We have $\theta_0 \not=\theta_{1}$, so $\alpha_m=0$ and $\beta_m=0$. This contradicts the definition of $m$, so the claim is proved.
By the claim, $X=\alpha_0 I + \alpha_1 A^*$ and $Y= \beta_0 I + \beta_1 A^*$. Using this to evaluate $A^*A=AX+Y$, we obtain
 \eqref{lem:five} with 
 \begin{align*}
 q=\alpha_1, \qquad  \alpha=\alpha_0, \qquad \beta = - \beta_1, \qquad \gamma=-\beta_0.
 \end{align*}
 We have shown that the sequence $q, \alpha, \beta, \gamma$ exists. This sequence is unique, 
because the following maps are linearly independent by Lemma \ref{lem:basisEE}(ii):
\begin{align*}
A A^*, \quad A, \quad A^*, \quad I.
\end{align*}
\end{proof}

\begin{definition}\label{def:profile} \rm The sequence $q, \alpha, \beta, \gamma $ from
Proposition \ref{prop:AAs} is called the {\it profile} of $A, A^*$.
\end{definition}

\begin{lemma} \label{lem:thRec} The following  {\rm (i), (ii)} hold for $0 \leq i \leq d$:
\begin{enumerate}
\item[\rm (i)] $q \theta_{i+1} = \theta_i +\beta$;
\item[\rm (ii)] $\theta^*_{i+1} = q \theta^*_i + \alpha$.
\end{enumerate}
\end{lemma}
\begin{proof} (i) In the equation \eqref{lem:five}, multiply each term on the left by $E_{i+1}$ and on the right by $E_i$. Simplify the result using
$E_{i+1} A^* E_i \not=0$. \\
\noindent (ii) In the equation \eqref{lem:five}, multiply each term on the left by $E^*_{i+1}$ and on the right by $E^*_i$. Simplify the result using
$E^*_{i+1} A E^*_i \not=0$. 
\end{proof}

\begin{lemma} \label{lem:ai} The following  {\rm (i), (ii)} hold for $0 \leq i \leq d$:
\begin{enumerate}
\item[\rm (i)] $a_i \bigl(\theta^*_i (q-1)+\alpha      \bigr) = \beta \theta^*_i + \gamma$;
\item[\rm (ii)] $a^*_i \bigl(\theta_i (1-q)  +\beta \bigr) = \alpha \theta_i - \gamma$.
\end{enumerate}
\end{lemma}
\begin{proof} (i) In the equation \eqref{lem:five}, multiply each term on the left and right by $E^*_{i}$. Simplify the result using
$E^*_{i} A E^*_i =a_i E^*_i$. \\
\noindent (ii) In the equation \eqref{lem:five}, multiply each term on the left and right by $E_i$. Simplify the result using
$E_{i} A^* E_i = a^*_i E_i$. 
\end{proof}

\begin{lemma} \label{lem:nonz} The scalars $\alpha$, $\beta$ satisfy the following inequalities.
\begin{enumerate}
\item[\rm (i)] Assume that $q\not=1$. Then 
\begin{align*}
    \alpha \not=(1-q)\theta^*_0,   \qquad \qquad    \beta \not=(q-1)\theta_0.
          \end{align*}
\item[\rm (ii)]  Assume that $q=1$. Then
\begin{align*} 
\alpha \not=0, \qquad \qquad \beta \not=0.
\end{align*}
\end{enumerate}
\end{lemma} 
\begin{proof} The inequality about $\alpha$ is from Lemma \ref{lem:thRec}(ii) with $i=0$ and $\theta^*_1 \not=\theta^*_0$. 
The inequality about $\beta$ is from Lemma \ref{lem:thRec}(i) with $i=0$ and $\theta_1 \not=\theta_0$. 
\end{proof}

\begin{lemma}\label{lem:pChar} For $0 \leq i \leq d$ we have
\begin{align} \label{eq:LR}
\frac{\theta_{i+1} - \theta_0}{\theta_1 - \theta_0}= \sum_{\ell=0}^i q^{-\ell}, \qquad \qquad 
\frac{\theta^*_{i+1} - \theta^*_0}{\theta^*_1 - \theta^*_0} = \sum_{\ell=0}^i q^\ell.
\end{align}
\end{lemma}
\begin{proof} We verify the equation on the left. Using Lemma \ref{lem:thRec}(i),
\begin{align*}
&\theta_{i+1} - q^{-i-1}\theta_0 \\
&= \theta_{i+1} - q^{-1} \theta_i + q^{-1}(\theta_i - q^{-1} \theta_{i-1}) + q^{-2}(\theta_{i-1} - q^{-1} \theta_{i-2}) + \cdots + q^{-i} (\theta_1 - q^{-1} \theta_0)
\\
&= (1 + q^{-1} + q^{-2} + \cdots + q^{-i}) q^{-1} \beta.
\end{align*}
\noindent  Observe that
\begin{align*}
\theta_{i+1}-\theta_0 &= \theta_{i+1} - q^{-i-1}\theta_0 + (q^{-i-1}-1)\theta_0 \\
&= \bigl(1+ q^{-1} + q^{-2} + \cdots + q^{-i} \bigr) \bigl( q^{-1} \beta + (q^{-1}-1) \theta_0\bigr) \\
&= \bigl(1+ q^{-1} + q^{-2} + \cdots + q^{-i} \bigr) (\theta_1 - \theta_0).
\end{align*}
This verifies the equation on the left in \eqref{eq:LR}. The equation on the right in \eqref{eq:LR} is similarly verified.
\end{proof}

\begin{lemma} \label{lem:sumZero} We have  $\sum_{\ell=0}^d q^\ell = 0$, and $\sum_{\ell=0}^i q^\ell \not=0$ for $0 \leq i \leq d-1$.
\end{lemma}
\begin{proof} We have $\theta^*_{d+1}=\theta^*_0$, and $\theta^*_{i+1} \not=\theta^*_0$ for $0 \leq i \leq d-1$. The result follows
in view of Lemma \ref{lem:pChar}.
\end{proof}

\noindent Recall $n=d+1$.

\begin{lemma} \label{lem:charF} The scalar $q$ is related to ${\rm Char}(\mathbb F)$ in the following way.
\begin{enumerate}
\item[\rm (i)] Assume that ${\rm Char}(\mathbb F)\not=n$. Then $q$ is a primitive $n^{\rm th}$ root of unity. 
\item[\rm (ii)] Assume that  ${\rm Char}(\mathbb F)=n$. Then $q=1$. 
\end{enumerate}
\end{lemma}
\begin{proof} First suppose that $q=1$. Then by Lemma \ref{lem:sumZero}, $n=0$ in $\mathbb F$ and $1,2,\ldots, d$ are nonzero in $\mathbb F$.  Therefore
${\rm Char}(\mathbb F) = n$. Next suppose that $q \not=1$. Then by Lemma \ref{lem:sumZero}, $q^n = 1$ and $q^j \not=1$ for $1 \leq j\leq d$.
Therefore $q$ is a primitive $n^{\rm th}$ root of unity. In this case ${\rm Char}(\mathbb F)\not=n$; otherwise
  $0 = q^n-1 = (q-1)^n$, forcing $q=1$ for a contradiction.
  The result follows from these comments.
  \end{proof}

\noindent In Definitions  \ref{def:eigval}, \ref{def:ai} and Proposition  \ref{prop:AAs}, we introduced various parameters that describe the circular bidiagonal pair $A, A^*$. Next, we consider how these parameters
are affected by an affine transformation of $A, A^*$.
\medskip

\noindent 
Pick scalars $s,s^*, t, t^*$ in $\mathbb F$ with $s, s^*$ nonzero. Consider the circular bidiagonal pair
\begin{align}
 A^\vee = s A + t I, \qquad \qquad (A^*)^\vee = s^* A^* + t^* I.
 \label{eq:st}
 \end{align}
 Note that $\lbrace E_i \rbrace_{i=0}^d$ and $\lbrace E^*_i \rbrace_{i=0}^d$ are orderings of the primitive idempotents of $A^\vee$
 and $(A^*)^\vee$, respectively. These orderings are standard. For $0 \leq i \leq d$  let $\theta^\vee_i$ (resp. $(\theta^*_i)^\vee$) denote the eigenvalue of $A^\vee$ (resp. $(A^*)^\vee$)
 for $E_i$ (resp. $E^*_i$). Also, define
 \begin{align*}
 a^\vee_i = {\rm tr}(A^\vee E^*_i), \qquad \qquad (a^*_i)^\vee = {\rm tr}\bigl((A^*)^\vee E_i\bigr).
 \end{align*}
 Let $q^\vee, \alpha^\vee, \beta^\vee, \gamma^\vee$ denote the profile of $A^\vee, (A^*)^\vee$.

\begin{lemma} \label{lem:affineAdj} We refer to the circular bidiagonal pair $A^\vee, (A^*)^\vee$ from \eqref{eq:st}.
 In the tables below, we describe various parameters for $A^\vee, (A^*)^\vee$ in terms of the corresponding parameters for $A, A^*$.
\bigskip

\centerline{
\begin{tabular}[t]{c|c}
   {\rm parameter } & {\rm parameter description}
\\
\hline
$\theta_i^\vee$ &  $s \theta_i + t$
\\
$(\theta^*_i)^\vee$ &  $s^* \theta^*_i + t^*$ \\
$a_i^\vee$ &  $s a_i + t$ \\
$(a^*_i)^\vee$ &  $s^* a^*_i + t^*$ 
\end{tabular}}
\bigskip

\centerline{
\begin{tabular}[t]{c|c}
   {\rm parameter } & {\rm parameter description}
\\
\hline
$q^\vee$ &  $q$
\\
$\alpha^\vee$ &  $s^* \alpha + t^*(1-q)$ \\
$\beta^\vee$ &  $s \beta + t(q-1)$ \\
$\gamma^\vee$ &  $s s^* \gamma + t s^* \alpha - s t^* \beta + t t^* (1-q)$ 
\end{tabular}}
\bigskip

 \end{lemma} 
\begin{proof} The first table is verified using Definitions \ref{def:eigval}, \ref{def:ai} and the discussion below \eqref{eq:st}.
 The second table is verified using Proposition  \ref{prop:AAs} and the comment above Lemma \ref{lem:affineAdj}.
\end{proof}

\noindent Next, we define what it means for the circular bidiagonal pair $A, A^*$ to be normalized. 

\begin{definition} \label{def:norm} \rm Let $E$ (resp. $E^*$) denote a primitive idempotent of $A$ (resp. $A^*$). Let $\theta$ (resp. $\theta^*$) denote the
corresponding eigenvalue.
The circular bidiagonal pair $A, A^*$ is said to be {\it normalized with respect to $E$ and $E^*$} whenever $\alpha, \beta, \theta, \theta^*$
satisfy the requirements in the table below:
\bigskip

\centerline{
\begin{tabular}[t]{c|cccc}
   {\rm case} & $\alpha $ & $\beta$ & $\theta $ & $\theta^*$
\\
\hline
${\rm Char}(\mathbb F)\not=n$ &  
$0$ & $0$ & $1 $
&$1$
\\
${\rm Char}(\mathbb F)=n$ &  
$1$ & $1$ & $0 $
&$0$
\end{tabular}}
\bigskip

\end{definition}
\noindent Next, we put  $A, A^*$ in normalized form by applying an affine transformation. 

\begin{lemma} \label{prop:main} We normalize the circular bidiagonal pair $A, A^*$ as follows.
\begin{enumerate}
\item[\rm (i)] Assume that ${\rm Char}(\mathbb F)\not=n$. Then the circular bidiagonal pair
\begin{align*}
 \frac{(q-1)A-\beta I}{(q-1)\theta_0 - \beta}, \qquad \qquad \frac{(1-q)A^* - \alpha I}{(1-q)\theta^*_0 - \alpha}
 \end{align*}
 is normalized with respect to $E_0$ and $E^*_0$.
 \item[\rm (ii)] Assume that ${\rm Char}(\mathbb F)=n$. Then the circular bidiagonal pair
\begin{align*}
 \frac{A - \theta_0 I}{\beta}, \qquad \qquad  \frac{A^*-\theta^*_0  I}{\alpha}
\end{align*}
is normalized  with respect to $E_0$ and $E^*_0$.
\end{enumerate}
\end{lemma}
\begin{proof} This is readily checked using Lemmas \ref{lem:nonz}, \ref{lem:charF},  \ref{lem:affineAdj} and Definition \ref{def:norm}.
\end{proof} 


\begin{lemma} \label{lem:sofar}  Assume that the circular bidiagonal pair $A, A^*$ is normalized  with respect to $E_0$ and $E^*_0$. 
\begin{enumerate}
\item[\rm (i)] Assume that ${\rm Char}(\mathbb F)\not=n$. Then
\begin{align*}
\frac{qA A^*-A^*A}{q-1} = \varepsilon I,
\end{align*}
where $\varepsilon = \gamma/(q-1)$. Moreover
\begin{align*}
\theta_i = q^{-i}, \qquad \theta^*_i = q^i, \qquad a_i = q^{-i} \varepsilon, \qquad a^*_i = q^i \varepsilon, \qquad \qquad (0 \leq i \leq d).
\end{align*}
\item[\rm (ii)] Assume that ${\rm Char}(\mathbb F)=n$. Then 
\begin{align*}
 AA^*-A^*A+A-A^*= \gamma I.
\end{align*}
Moreover
\begin{align*}
\theta_i = i, \qquad \theta^*_i = i, \qquad a_i = i+\gamma, \qquad a^*_i = i-\gamma, \qquad \qquad (0 \leq i \leq d).
\end{align*}
\end{enumerate}
\end{lemma}
\begin{proof} Evaluate Proposition \ref{prop:AAs} and Lemmas \ref{lem:thRec}, \ref{lem:ai} using Definition \ref{def:norm}.
\end{proof}

\begin{lemma} \label{lem:notamong} Assume that the circular bidiagonal pair $A, A^*$ is normalized  with respect to $E_0$ and $E^*_0$.
\begin{enumerate}
\item[\rm (i)] Assume that ${\rm Char}(\mathbb F)\not=n$. Then the scalar $\varepsilon$ from Lemma \ref{lem:sofar}(i) is not among $1,q, q^2,\ldots, q^d$.
\item[\rm (ii)]  Assume that ${\rm Char}(\mathbb F)=n$. Then the scalar $\gamma$ from Lemma \ref{lem:sofar}(ii) is not among $0,1,2,\ldots, d$.
\end{enumerate}
\end{lemma}
\begin{proof} Use Lemma \ref{lem:match} and the data in Lemma \ref{lem:sofar}.
\end{proof} 

\begin{proposition} \label{prop:almostdone} 
Assume that the circular bidiagonal pair $A, A^*$ is normalized  with respect to $E_0$ and $E^*_0$. 
\begin{enumerate}
\item[\rm (i)] Assume that ${\rm Char}(\mathbb F)\not=n$. Then the circular bidiagonal pair $A, A^*$ is isomorphic to ${\rm CBP}(\mathbb F;d,q,\varepsilon)$,
where $q, \varepsilon$ are from Lemma \ref{lem:sofar}(i).
\item[\rm (ii)] Assume that ${\rm Char}(\mathbb F)=n$.  Then the circular bidiagonal pair $A, A^*$ is isomorphic to 
${\rm CBP}(\mathbb F;d,\gamma)$, where $\gamma$ is from Lemma \ref{lem:sofar}(ii).
\end{enumerate}
\end{proposition}
\begin{proof} (i) 
By Lemma \ref{lem:charF}(i), $q$ is a primitive $n^{\rm th}$ root of unity. By Lemma \ref{lem:notamong}(i), $\varepsilon$ is not among $1,q,q^2, \ldots, q^d$. 
The circular bidiagonal pair ${\rm CBP}(\mathbb F;d,q,\varepsilon)$ is described in Lemma \ref{ex:ddq}.
We have $\theta_0=1$ by Lemma \ref{lem:sofar}(i). Therefore $1$ is an eigenvalue of $A$;
let $0 \not=\xi\in V$ denote a corresponding eigenvector. Consider the basis $\lbrace E^*_i \xi \rbrace_{i=0}^d$ of $V$ from Lemma \ref{lem:BBasis}.
 Let $B$ (resp. $B^*$) denote the matrix in ${\rm Mat}_{d+1}(\mathbb F)$ that represents $A$ (resp. $A^*$) with respect to $\lbrace E^*_i \xi \rbrace_{i=0}^d$.
Using Lemma \ref{lem:BBasis2} (with $\theta=1$) and the data in Lemma \ref{lem:sofar}(i), we find that ${\rm CBP}(\mathbb F;d,q,\varepsilon)$ 
is equal to $B, B^*$. The result follows.
 \\ 
        \noindent (ii) By Lemma \ref{lem:charF}(ii),  ${\rm Char}(\mathbb F)=n$.  By Lemma \ref{lem:notamong}(ii),
$\gamma$ is not among $0,1,2,\ldots, d$. 
The circular bidiagonal pair ${\rm CBP}(\mathbb F;d, \gamma)$ is described in Lemma \ref{ex:dd}.
We have $\theta_0=0$ by Lemma \ref{lem:sofar}(ii). Therefore $0$ is an eigenvalue of $A$;
let $0 \not=\xi\in V$ denote a corresponding eigenvector. Consider the basis $\lbrace E^*_i \xi \rbrace_{i=0}^d$ of $V$ from Lemma  \ref{lem:BBasis}.
 Let $B$ (resp. $B^*$) denote the matrix in ${\rm Mat}_{d+1}(\mathbb F)$ that represents $A$ (resp. $A^*$) with respect to  $\lbrace E^*_i \xi \rbrace_{i=0}^d$.
Using Lemma \ref{lem:BBasis2} (with $\theta=0$) and the data in Lemma \ref{lem:sofar}(ii), we find that ${\rm CBP}(\mathbb F;d, \gamma)$ 
is equal to $B, B^*$. The result follows.
\end{proof}

\noindent Theorem \ref{thm:main} is immediate from Lemma \ref{prop:main} and Proposition \ref{prop:almostdone}.

\section{The proof of Lemma \ref{thm:th2}}

\noindent In this section, we prove Lemma \ref{thm:th2}.
\medskip

\noindent {\it Proof of Lemma \ref{thm:th2}} (i) Assume that
 ${\rm CBP}(\mathbb F; d, q, \varepsilon)$ and ${\rm CBP}(\mathbb F; d, q', \varepsilon')$ are isomorphic. We will show that $q=q'$ and $\varepsilon = \varepsilon'$.
Write $A, A^*$ for
 ${\rm CBP}(\mathbb F; d, q, \varepsilon)$ and $B, B^*$ for ${\rm CBP}(\mathbb F; d, q', \varepsilon')$.
By Definition \ref{def:cbpIso},
 there exists an invertible $P \in {\rm Mat}_{d+1}(\mathbb F)$
such that $PA=BP$ and $PA^*=B^* P$. The matrix $A^*$ is diagonal, and its diagonal entries are mutually distinct.
The  matrix $B^*$ is diagonal, and its diagonal entries are mutually distinct.
Examining the entries of $PA^*= B^*P$, we find that there exists a permutation $p$ of the set
$\lbrace 0,1,2,\ldots, d\rbrace$ such that for $0 \leq i,j\leq d$, 
\begin{align*}
 j = p(i)\quad \Leftrightarrow \quad P_{i,j} \not=0 \quad \Leftrightarrow \quad A^*_{j,j} = B^*_{i,i}.
 \end{align*}
We have $A^*_{0,0}=1=B^*_{0,0}$.  Therefore $p(0)=0$. Next we show that $P$ is diagonal. The matrices $A$ and $B$ are circular bidiagonal. 
For $1 \leq i \leq d$ we examine the $\bigl(i,p(i-1)\bigr)$-entry in $PA=BP$; this gives
\begin{align*}
P_{i,p(i)} A_{p(i),p(i-1)} = B_{i,i-1} P_{i-1,p(i-1)}.
\end{align*}
\noindent By construction $B_{i,i-1}\not=0$ and $P_{i-1,p(i-1)}\not=0$. Therefore $A_{p(i), p(i-1)} \not=0$. Of course $p(i) \not=p(i-1)$, so 
$p(i)=p(i-1)+1$. Using induction and $p(0)=0$, we obtain $p(i)=i$ for $0 \leq i \leq d$. We have shown that $P$ is diagonal. Consequently $P$ commutes
with $A^*$, so $A^*=B^*$. Therefore $q=q'$. Also, for $0 \leq i \leq d$ we have $A_{i,i}=B_{i,i}$, which implies that $\varepsilon=\varepsilon'$.
We are done in one logical direction. We now consider the opposite logical direction. Assume that $q=q' $ and $\varepsilon = \varepsilon'$.
Then
 ${\rm CBP}(\mathbb F; d, q, \varepsilon)$ and ${\rm CBP}(\mathbb F; d, q', \varepsilon')$ are the same, and hence isomorphic.
\\
\noindent (ii) Similar to the proof of (i). \hfill $\Box$

\section{The proof of Lemma \ref{thm:th3}}

\noindent In this section, we prove Lemma \ref{thm:th3}. 
\medskip

\noindent We begin with some comments about the matrices $A=A(q, \varepsilon)$ and $A^*=A^*(q)$ from Lemma \ref{ex:ddq}.

\begin{lemma} \label{lem:p1}
\label{lem:add1} With the above notation, 
\begin{align*}
(A^*-\varepsilon I) A(q, q\varepsilon) = q A(q, \varepsilon) (A^*-\varepsilon I).
\end{align*}
\end{lemma}
\begin{proof} This is routinely verified by matrix multiplication, using the data in Lemma \ref{ex:ddq}.
\end{proof}

\begin{lemma} \label{lem:add2}  The map $A^*-\varepsilon I$ from Lemma \ref{lem:p1} is an isomorphism of circular bidiagonal pairs,  from
$A(q, q\varepsilon), A^*$ to $q A(q, \varepsilon), A^*$.
\end{lemma}
\begin{proof} Define the matrix $P=A^*-\varepsilon I$. The matrix $P$ is invertible, since
$\varepsilon$ is not among $1,q,q^2,\ldots, q^d$. By Lemma \ref{lem:add1} and the construction,
\begin{align*}
P A(q, q \varepsilon) = q A(q, \varepsilon) P, \qquad \qquad P A^* = A^* P.
\end{align*}
The result follows in view of Definition \ref{def:cbpIso}.
\end{proof} 

\begin{corollary}\label{cor:add3}
The  circular bidiagonal pairs ${\rm CBP}(\mathbb F; d, q, \varepsilon)$ and
${\rm CBP}(\mathbb F; d, q, q\varepsilon)$ 
are affine equivalent.
\end{corollary}
\begin{proof} By Definitions \ref{def:name1}, \ref{def:aff} and Lemma \ref{lem:add2}.
\end{proof}

\begin{corollary} \label{cor:n1} The following
circular bidiagonal pairs are mutually affine equivalent:
\begin{align*}
{\rm CBP}(\mathbb F; d, q, q^i \varepsilon ) \qquad\qquad i \in \lbrace 0,1,2,\ldots, d\rbrace.
\end{align*}
\end{corollary}
\begin{proof} By Corollary \ref{cor:add3}, and since affine equivalence is an equivalence relation.
\end{proof}

\noindent Next, we have some comments about the matrices $A=A(\gamma)$ and $A^*$ from Lemma \ref{ex:dd}.

\begin{lemma}\label{lem:ad1} With the above notation,
\begin{align*}
(A^*+\gamma I) A(\gamma-1) = \bigl( A(\gamma)-I\bigr) (A^*+\gamma I).
\end{align*}
\end{lemma}
\begin{proof} This is routinely verified by matrix multiplication, using the data in Lemma \ref{ex:dd}.
\end{proof}

\begin{lemma}\label{lem:ad2} The map $A^*+\gamma I$ from Lemma \ref{lem:ad1} is an isomorphism of circular bidiagonal pairs,
from $A(\gamma-1), A^*$ to $A(\gamma)-I, A^*$.
\end{lemma}
\begin{proof} Define the matrix $P=A^*+ \gamma I$. The matrix $P$ is invertible, since $\gamma$ is not among $0,1,2,\ldots, d$. By Lemma \ref{lem:ad1} and the construction,
\begin{align*}
P A(\gamma-1) = \bigl( A(\gamma)-I\bigr)  P, \qquad \qquad P A^* = A^* P.
\end{align*}
The result follows in view of Definition \ref{def:cbpIso}.
\end{proof}

\begin{corollary} \label{cor:ad3}
The  circular bidiagonal pairs ${\rm CBP}(\mathbb F; d, \gamma)$ and
${\rm CBP}(\mathbb F; d, \gamma-1)$ 
are affine equivalent.
\end{corollary}
\begin{proof} By Definitions \ref{def:name2}, \ref{def:aff} and Lemma \ref{lem:ad2}.
\end{proof}

\begin{corollary} \label{cor:n2} The following
circular bidiagonal pairs are mutually affine equivalent:
\begin{align*}
{\rm CBP}(\mathbb F; d, \gamma + i) \qquad\qquad i \in \lbrace 0,1,2,\ldots, d\rbrace.
\end{align*}
\end{corollary}
\begin{proof} By Corollary \ref{cor:ad3}, and since affine equivalence is an equivalence relation.
\end{proof}

\noindent {\it Proof of Lemma \ref{thm:th3}} (i) Assume that
 ${\rm CBP}(\mathbb F; d, q, \varepsilon)$ and ${\rm CBP}(\mathbb F; d, q', \varepsilon')$ are affine equivalent. We will show that $q=q'$ and 
 $\varepsilon'  \in \lbrace  \varepsilon, q \varepsilon, q^2 \varepsilon, \ldots, q^d \varepsilon \rbrace$.
Write $A, A^*$ for  ${\rm CBP}(\mathbb F; d, q, \varepsilon)$ and $B, B^*$ for  ${\rm CBP}(\mathbb F; d, q', \varepsilon')$.
The profile of $A,A^*$ is  $q, \alpha, \beta, \gamma$ where
\begin{align} \label{eq:abc}
\alpha = 0, \qquad \quad \beta=0, \qquad \quad \gamma = \varepsilon (q-1).
\end{align}
The profile of $B,B^*$ is $q', \alpha', \beta', \gamma'$ 
where
\begin{align} \label{eq:apbpcp}
\alpha' = 0, \qquad \quad \beta'=0, \qquad \quad \gamma' = \varepsilon' (q'-1).
\end{align}
By assumption, there exist scalars $s, s^*, t, t^*$ in $\mathbb F$ with $s$, $s^*$ nonzero such that $sA+tI, s^*A^*+t^* I$ is isomorphic to $B, B^*$.
Define $A^\vee = sA+tI$ and $(A^*)^\vee = s^*A^*+t^* I$. The  profile $q^\vee$, $\alpha^\vee$, $\beta^\vee$, $\gamma^\vee$ of  $A^\vee, (A^*)^\vee$ 
is described in Lemma \ref{lem:affineAdj}.
The circular bidiagonal pairs $A^\vee, (A^*)^\vee$ and $B, B^*$ are isomorphic, so they have the same profile:
\begin{align*}
q^\vee = q', \qquad \quad \alpha^\vee = \alpha', \qquad \quad \beta^\vee = \beta', \qquad \quad \gamma^\vee = \gamma'.
\end{align*}
Evaluate the above equations using \eqref{eq:abc}, \eqref{eq:apbpcp} and the second table in Lemma \ref{lem:affineAdj}. This yields
\begin{align*}
q=q', \qquad \quad t = 0, \qquad \quad t^*=0, \qquad \quad  s s^* \varepsilon =\varepsilon'.
\end{align*}
By construction, the circular bidiagonal pairs  $sA, s^*A^*$ and $B, B^*$ are isomorphic. Therefore $sA$ has the same eigenvalues as $B$, and
 $s^*A^*$ has the same eigenvalues as $B^*$.
The scalar $1$ is an eigenvalue of $A$, so the scalar $s$ is an eigenvalue of $sA$. The eigenvalues of $B$ are $1,q,q^2,\ldots, q^d$.
By these comments $s \in \lbrace 1,q, q^2, \ldots, q^d\rbrace$.
The scalar $1$ is an eigenvalue of $A^*$, so the scalar $s^*$ is an eigenvalue of $s^*A^*$. The eigenvalues of $B^*$ are $1,q,q^2,\ldots, q^d$.
By these comments $s^* \in \lbrace 1,q, q^2, \ldots, q^d\rbrace$.  We may now argue
\begin{align*}
\varepsilon'  = s s^* \varepsilon \in \lbrace  \varepsilon, q \varepsilon, q^2 \varepsilon, \ldots, q^d \varepsilon \rbrace.
\end{align*}

\noindent Next, we reverse the logical direction. Assume that $q'=q$ and $\varepsilon'\in 
\lbrace \varepsilon, q \varepsilon, q^2 \varepsilon, \ldots, q^d \varepsilon\rbrace$. Then  ${\rm CBP}(\mathbb F; d, q, \varepsilon)$ and ${\rm CBP}(\mathbb F; d, q', \varepsilon')$ are affine equivalent by Corollary \ref{cor:n1}.
\\
\noindent (ii) Assume that
 ${\rm CBP}(\mathbb F; d, \gamma)$ and ${\rm CBP}(\mathbb F; d, \gamma')$ are affine equivalent. We will show that
 $\gamma' -\gamma \in \lbrace 0,1,2,\ldots,d \rbrace$.
Write $A, A^*$ for ${\rm CBP}(\mathbb F; d, \gamma)$ and
$B, B^*$ for  ${\rm CBP}(\mathbb F; d, \gamma')$.
The profile of $A,A^*$ is  $q, \alpha, \beta, \gamma$ where
\begin{align} \label{eq:abc1}
q=1, \qquad \quad \alpha = 1, \qquad \quad \beta=1.
\end{align}
The profile of $B, B^*$ is  $q', \alpha', \beta', \gamma'$ where
\begin{align} \label{eq:apbpcp1}
q'=1, \qquad \quad \alpha' = 1, \qquad \quad \beta'=1.
\end{align}
By assumption, there exist scalars $s, s^*, t, t^*$ in $\mathbb F$ with $s$, $s^*$ nonzero such that $sA+tI, s^*A^*+t^* I$ is isomorphic to $B, B^*$.
Define $A^\vee = sA+tI$ and $(A^*)^\vee = s^*A^*+t^* I$. The profile  $q^\vee$, $\alpha^\vee$, $\beta^\vee$, $\gamma^\vee$ of  $A^\vee, (A^*)^\vee$
is described in Lemma \ref{lem:affineAdj}.
The circular bidiagonal pairs $A^\vee, (A^*)^\vee$ and $B, B^*$ are isomorphic, so they have the same profile:
\begin{align*}
q^\vee = q', \qquad \quad \alpha^\vee = \alpha', \qquad \quad \beta^\vee = \beta', \qquad \quad \gamma^\vee = \gamma'.
\end{align*}
Evaluate the above equations using \eqref{eq:abc1}, \eqref{eq:apbpcp1} and the second table in Lemma \ref{lem:affineAdj}. This yields
\begin{align*}
 s=1, \qquad \quad s^*=1, \qquad \quad  \gamma'-\gamma=t-t^*.
\end{align*}
By construction, the circular bidiagonal pairs  $A+tI, A^*+t^*I$ and $B, B^*$ are isomorphic. Therefore $A+tI$ has the same eigenvalues as $B$, and
 $A^*+t^*I$ has the same eigenvalues as $B^*$.
The scalar $0$ is an eigenvalue of $A$, so the scalar $t$ is an eigenvalue of $A+tI$. The eigenvalues of $B$ are $0,1,2,\ldots, d$.
By these comments $t \in \lbrace 0,1,2,\ldots, d\rbrace$.
The scalar $0$ is an eigenvalue of $A^*$, so the scalar $t^*$ is an eigenvalue of $A^*+t^*I$. The eigenvalues of $B^*$ are $0,1,2,\ldots, d$.
By these comments $t^* \in \lbrace 0,1,2,\ldots, d\rbrace$.  We may now argue
\begin{align*}
\gamma'  - \gamma = t-t^* \in  \lbrace 0,1,2,\ldots, d\rbrace.
\end{align*}
\noindent Next, we reverse the logical direction. Assume that $\gamma' - \gamma \in \lbrace 0,1,2,\ldots, d\rbrace$.
Then  ${\rm CBP}(\mathbb F; d,  \gamma)$ and ${\rm CBP}(\mathbb F; d,  \gamma')$ are affine equivalent by Corollary \ref{cor:n2}.
\hfill $\Box$

\section{Isomorphism and duality}

\noindent For circular bidiagonal pairs,  the concepts of duality and isomorphism were explained in Definitions
\ref{def:dual} and \ref{def:cbpIso}, respectively. In this section, we use these concepts to interpret the proof of Lemmas \ref{ex:ddq}, \ref{ex:dd}.

\begin{proposition} \label{prop:dual1}
We refer  to the circular bidiagonal pair ${\rm CBP}(\mathbb F; d, q, \varepsilon)$ in Lemma \ref{ex:ddq}. 
The matrix  $P(q,\varepsilon)$ from  \eqref{eq:Pmatq} is an isomorphism of
circular bidiagonal pairs from $A(q^{-1}, \varepsilon), A^*(q^{-1})$ to $A^*(q), A(q, \varepsilon)$.
\end{proposition} 
\begin{proof} We showed in the proof of Lemma \ref{ex:ddq} that $P(q, \varepsilon)$ is invertible.
The result follows from this along with \eqref{eq:AP1}, \eqref{eq:AP2}
and
Definition \ref{def:cbpIso}.
\end{proof}

\begin{corollary}\label{cor:dual} The following are dual, up to isomorphism of circular bidiagonal pairs:
\begin{align*}
{\rm CBP}(\mathbb F; d, q, \varepsilon); \qquad \qquad {\rm CBP}(\mathbb F; d, q^{-1}; \varepsilon).
\end{align*}
\end{corollary}
\begin{proof} By Definition \ref{def:dual} and Proposition \ref{prop:dual1}.
\end{proof} 

\begin{proposition} \label{prop:dual2}
We refer  to the circular bidiagonal pair ${\rm CBP}(\mathbb F; d, \gamma)$ in Lemma \ref{ex:dd}. The matrix  $P(\gamma)$ from \eqref{eq:Pmat}
 is an isomorphism of
circular bidiagonal pairs from $A(-\gamma), A^*$ to $A^*, A(\gamma)$.
\end{proposition} 
\begin{proof} We showed in the proof of Lemma \ref{ex:dd} that $P(\gamma)$ is invertible.
The result follows from this along with  \eqref{eq:APP2}, \eqref{eq:AsP2} and  Definition \ref{def:cbpIso}.
\end{proof}

\begin{corollary} \label{cor:dual2}
The following are dual, up to isomorphism of circular bidiagonal pairs:
\begin{align*}
{\rm CBP}(\mathbb F; d, \gamma), \qquad \qquad {\rm CBP}(\mathbb F; d, -\gamma).
\end{align*}
\end{corollary}
\begin{proof} By Definition \ref{def:dual} and Proposition \ref{prop:dual2}.
\end{proof}

\section{The proof of Lemma \ref{lem:com}}
\noindent Our goal in this section is to prove Lemma \ref{lem:com}. Our proof strategy is to display the isomorphism involved. We will give a detailed description of this isomorphism.

\medskip

\noindent Recall the circular bidiagonal pair ${\rm CBP}(\mathbb F; d,q,\varepsilon)$ from Lemma \ref{ex:ddq}. 

\begin{definition} \label{def:raiseq} \rm
Referring to  ${\rm CBP}(\mathbb F; d,q,\varepsilon)$,
define a matrix $R=R(q,\varepsilon)$ in ${\rm Mat}_{d+1}(\mathbb F)$ with entries 
$R_{0,d}=1-\varepsilon$ and $R_{i,i-1} = q^i-\varepsilon$ for $1 \leq i \leq d$. All other entries of  $R$ are zero.
We call $R$ the {\it raising matrix} for ${\rm CBP}(\mathbb F; d,q,\varepsilon)$.
\end{definition}

\begin{example}\rm Referring to Definition \ref{def:raiseq},  assume that $d=4$. Then
\begin{align*}
R= \begin{pmatrix} 0 &0&0&0&1-\varepsilon \\ q-  \varepsilon&0&0&0&0 \\ 0&q^2- \varepsilon& 0&0&0 \\ 0&0&q^3- \varepsilon&
0&0 \\ 0&0&0&q^4-\varepsilon& 0
\end{pmatrix}.
\end{align*}
\end{example}

\begin{lemma} \label{lem:raise1q} With reference to Definition \ref{def:raiseq},
 the following {\rm (i), (ii)} hold:
\begin{enumerate} 
\item[\rm (i)] $R^{d+1} = (-1)^d (\varepsilon;q)_{d+1} I$;
\item[\rm (ii)] $R$ is invertible.
\end{enumerate}
\end{lemma}
\begin{proof} (i) By matrix multiplication. \\
 \noindent (ii) By (i) and since $(\varepsilon;q)_{d+1}\not=0$. \\
\end{proof}

\noindent Next, we explain how $R$ is related to the matrices $A=A(q,\varepsilon)$ and $A^*=A^*(q)$ from Lemma \ref{ex:ddq}.
\begin{lemma} \label{lem:raise2q} With the above notation, we have
\begin{enumerate}
\item[\rm (i)] $A^*A-\varepsilon I = R = q (AA^*-\varepsilon I)$;
\item[\rm (ii)]  $q A R= R A$ and $q^{-1} A^* R = R A^*$.
\end{enumerate}
 \end{lemma}
 \begin{proof} (i) By matrix multiplication, using the data in Lemma \ref{ex:ddq} and Definition \ref{def:raiseq}. \\
 \noindent (ii) Observe that
 \begin{align*}
 &qAR-RA = qA(A^*A-\varepsilon I) - q(A A^*-\varepsilon I ) A = 0; \\
 & q^{-1} A^* R - R A^* = q^{-1} A^* q(AA^*-\varepsilon I) - (A^* A - \varepsilon I) A^* = 0.
 \end{align*}
\end{proof}

\noindent Next, we explain how $R$ is related to the primitive idempotents of $A$ and $A^*$. 
For $0 \leq i \leq d$, let $E_i$ (resp. $E^*_i$) denote the primitive idempotent of $A$ (resp. $A^*$) for the
eigenvalue $ q^{-i}$ (resp. $q^i$). 

\begin{lemma} \label{lem:raise3q} With the above notation, we have
\begin{enumerate}
\item[\rm (i)] $RE_i=E_{i+1} R$ and $RE^*_i=E^*_{i+1}R$ for $0 \leq i \leq d$;
\item[\rm (ii)] $RE_i V= E_{i+1} V$ and $R E^*_iV = E^*_{i+1}V$ for $0 \leq i \leq d$.
\end{enumerate}
 \end{lemma}
 \begin{proof} (i) To obtain  $RE_i=E_{i+1} R$, 
 multiply each side of \eqref{eq:defEi} on the left by $R$ and on the right by $R^{-1}$. Evaluate
 the result using $\theta_r = q^{-r}$ $(0 \leq r \leq d)$ along with $qA=RAR^{-1}$. The equation $RE^*_i=E^*_{i+1}R$ is similar obtained.
 \\
 \noindent (ii) By (i) above and Lemma \ref{lem:raise1q}(ii).
 \end{proof}

 \begin{proposition} \label{prop:raiseq}
With the above notation, $R$  is an isomorphism of circular bidiagonal pairs from $A, A^*$ to  $qA, q^{-1}A^*$.
\end{proposition}
 \begin{proof} By   Definition \ref{def:cbpIso} and Lemma \ref{lem:raise2q}(ii).
 \end{proof}
 

\noindent We turn our attention to the circular bidiagonal pair ${\rm CBP}(\mathbb F;d,\gamma)$ from Lemma \ref{ex:dd}.

\begin{definition} \label{def:raise} \rm Referring to ${\rm CBP}(\mathbb F;d,\gamma)$,
define a matrix $R=R(\gamma)$ in ${\rm Mat}_{d+1}(\mathbb F)$ with entries 
$R_{0,d}=\gamma$ and $R_{i,i-1} = i+\gamma$ for $1 \leq i \leq d$. All other entries of $R$ are zero.
We call $R$ the {\it raising matrix} for ${\rm CBP}(\mathbb F;d,\gamma)$.
\end{definition}

\begin{example}\rm Referring to Definition \ref{def:raise},  assume that $d=4$. Then
\begin{align*}
R= \begin{pmatrix} 0 &0&0&0&\gamma \\ 1+  \gamma&0&0&0&0 \\ 0&2+ \gamma& 0&0&0 \\ 0&0&3+\gamma&
0&0 \\ 0&0&0&4+\gamma& 0
\end{pmatrix}.
\end{align*}
\end{example}

\begin{lemma} \label{lem:raise1} With reference to Definition \ref{def:raise}, the following {\rm (i), (ii)} hold:
\begin{enumerate} 
\item[\rm (i)] $R^{d+1} = (\gamma)_{d+1} I$;
\item[\rm (ii)] $R$ is invertible.
\end{enumerate}
 \end{lemma}
 \begin{proof} (i) By matrix multiplication. \\
 \noindent (ii) By (i) and since $(\gamma)_{d+1} \not=0$.
 \end{proof}

\noindent Next, we describe how $R$ is related to the matrices $A=A(\gamma)$ and $A^*$ from Lemma \ref{ex:dd}.
\begin{lemma} \label{lem:raise2} With the above notation,
\begin{enumerate}
\item[\rm (i)] $A^*-A+\gamma I = R = A A^*-A^* A$;
\item[\rm (ii)]  $(A-I)R= R A$ and $(A^*-I) R = R A^*$.
\end{enumerate}
\end{lemma}
\begin{proof} (i) By matrix multiplication, using the data in Lemma \ref{ex:dd} and Definition \ref{def:raise}. \\
 \noindent (ii) We have
 \begin{align*}
 &\lbrack A, R\rbrack = \lbrack A, A^*-A+\gamma I\rbrack = \lbrack A, A^* \rbrack = R, \\
 &\lbrack A^*, R\rbrack = \lbrack A^*, A^*-A+\gamma I\rbrack =-\lbrack A^*, A \rbrack =  \lbrack A, A^* \rbrack = R.
 \end{align*}
\end{proof}

\noindent Next, we describe how $R$ is related to the primitive idempotents of $A$ and $A^*$.
For $0 \leq i \leq d$, let $E_i$ (resp. $E^*_i$) denote the primitive idempotent of $A$ (resp. $A^*$) for the
eigenvalue $i$.
\begin{lemma} \label{lem:raise3} With the above notation,
\begin{enumerate}
\item[\rm (i)] $RE_i=E_{i+1} R$ and $RE^*_i=E^*_{i+1}R$ for $0 \leq i \leq d$;
\item[\rm (ii)] $RE_i V= E_{i+1} V$ and $R E^*_iV = E^*_{i+1}V$ for $0 \leq i \leq d$.
\end{enumerate}
 \end{lemma}
 \begin{proof} (i) To obtain  $RE_i=E_{i+1} R$, multiply each side of \eqref{eq:defEi} on the left by $R$ and on the right by $R^{-1}$. Evaluate
 the result using $\theta_r = r$ $(0 \leq r \leq d)$ along with $A-I=RAR^{-1}$. The equation $RE^*_i=E^*_{i+1}R$ is similar obtained.
 \\
 \noindent (ii) By (i) and Lemma \ref{lem:raise1}(ii).
 \end{proof}

 \begin{proposition} \label{prop:raise2}
With the above notation, $R$ is an isomorphism of circular bidiagonal pairs from $A, A^*$ to  $A-I, A^*-I$.
 \end{proposition}
 \begin{proof} By   Definition \ref{def:cbpIso} and Lemma \ref{lem:raise2}(ii).
 \end{proof}
 
 \noindent Lemma \ref{lem:com} is immediate from Propositions \ref{prop:raiseq} and \ref{prop:raise2}.


\section{Circular Hessenberg pairs}

In \cite{jhl} Jae-ho Lee introduced the concept of a circular Hessenberg pair.
A circular bidiagonal pair is a special case of a circular Hessenberg pair. 
In Lemmas \ref{ex:ddq} and \ref{ex:dd}, we gave some examples of a circular bidiagonal pair. In the present section, we describe these examples
using the notation of \cite{jhl}. 
\medskip

\noindent 
In \cite{jhl} it is assumed that $d\geq 3$; we make the same assumption throughout this section.

\begin{example} \label{ex:chp1} \rm Recall ${\rm CBP}(\mathbb F; d, q, \varepsilon)$ from Lemma \ref{ex:ddq}. This corresponds to \cite[Example~5.1]{jhl} with parameters
\begin{align*}
a=0, \quad b=0, \quad c=1, \quad a^*=0, \quad b^*=1, \quad c^*=0, \quad y = 1-\varepsilon, \quad z=0.
\end{align*}
Here are some related parameters. Referring to \cite[Example~5.1]{jhl}, 
\begin{align*}
&\theta_i = q^{-i}, \qquad \theta^*_i = q^i  \qquad (0 \leq i \leq d); \\
&\phi_i = (q^i-1)(q^i - \varepsilon), \qquad \vartheta_i = (q^i-1)(1-\varepsilon) \qquad (1 \leq i \leq d).
\end{align*}
\noindent Referring to \cite[Proposition~6.12]{jhl},
\begin{align*}
&b_i = 0 \quad (0 \leq i \leq d-1); \qquad
a_i = q^{-i} \varepsilon \quad (0 \leq i \leq d); \\
&c_i = 1-q^{-i} \varepsilon \quad (1 \leq i \leq d); \qquad
\xi = 1-\varepsilon.
\end{align*}
\noindent Referring to \cite[Proposition~6.11]{jhl},
\begin{align*}
&b^*_i = 0 \quad (0 \leq i \leq d-1); \qquad
a^*_i = q^{i} \varepsilon \quad (0 \leq i \leq d); \\
&c^*_i = 1-q^{i} \varepsilon \quad (1 \leq i \leq d); \qquad
\xi^*= 1-\varepsilon.
\end{align*}
\end{example}

\begin{example}\label{ex:chp2} \rm Recall  ${\rm CBP}(\mathbb F; d, \gamma)$ from Lemma \ref{ex:dd}. This corresponds to \cite[Example~5.2]{jhl} with parameters
\begin{align*}
a=0, \quad b=1, \quad c=0, \quad a^*=0, \quad b^*=1, \quad c^*=0, \quad y = -\gamma, \quad z=0.
\end{align*}
Here are some related parameters. Referring to \cite[Example~5.2]{jhl}, 
\begin{align*}
&\theta_i = i, \qquad \theta^*_i = i  \qquad (0 \leq i \leq d); \\
&\phi_i = -i (i+\gamma), \qquad \vartheta_i = -i\gamma \qquad (1 \leq i \leq d).
\end{align*}
\noindent Referring to \cite[Proposition~6.12]{jhl},
\begin{align*}
&b_i = 0 \quad (0 \leq i \leq d-1); \qquad
a_i = i+\gamma \quad (0 \leq i \leq d); \\
&c_i = -i-\gamma \quad (1 \leq i \leq d); \qquad
\xi = -\gamma.
\end{align*}
\noindent Referring to \cite[Proposition~6.11]{jhl},
\begin{align*}
&b^*_i = 0 \quad (0 \leq i \leq d-1); \qquad
a^*_i = i-\gamma \quad (0 \leq i \leq d); \\
&c^*_i = \gamma-i \quad (1 \leq i \leq d); \qquad
\xi^*= \gamma.
\end{align*}
\end{example}





\section{Acknowledgement} The first author thanks Jae-ho Lee for many conversations about
circular bidiagonal pairs and circular Hessenberg pairs. The authors thank
\v{S}tefko Miklavi\v{c} for giving this paper a close reading and offering valuable comments.


\bigskip

\noindent Paul Terwilliger \hfil\break
\noindent Department of Mathematics \hfil\break
\noindent University of Wisconsin \hfil\break
\noindent 480 Lincoln Drive \hfil\break
\noindent Madison, WI 53706-1388 USA \hfil\break
\noindent email: {\tt terwilli@math.wisc.edu }\hfil\break
\bigskip
 
\noindent  Arjana \v{Z}itnik \hfil\break
\noindent  Faculty of Mathematics and Physics \hfil\break
\noindent  University of Ljubljana, and IMFM \hfil\break
\noindent Jadranska 19, 1000 Ljubljana, Slovenia \hfil\break
\noindent email: {\tt Arjana.Zitnik@fmf.uni-lj.si} \hfil\break

 \end{document}